\documentclass{lms}
\usepackage{anysize}

\usepackage{amssymb}
\usepackage{amsmath,amsbsy,amssymb,amscd}
\usepackage{t1enc}
\usepackage[cp1250]{inputenc}
\usepackage[british]{babel}
\usepackage[all]{xy}
\usepackage{color}
\usepackage{amsfonts}
\usepackage{latexsym}
\usepackage{mathrsfs}
\usepackage{hyperref}
\usepackage{graphicx}

\usepackage{color}
\usepackage{enumerate}
\usepackage{bbm}
\definecolor{orange}{rgb}{1,0.5,0}

\usepackage{mathrsfs} %Script dla zapisu \mathscr{}

\DeclareMathAlphabet{\mathpzc}{OT1}{pzc}{L}{it} %stylizowany
\newtheorem{theorem}{Theorem}[section] % 1st argument is your name for it
\newtheorem{lemma}[theorem]{Lemma}     % 2nd argument is what is printed
\newtheorem{corollary}[theorem]{Corollary}

\newnumbered{assertion}{Assertion}    % 1st argument is your name for it
\newnumbered{conjecture}{Conjecture}  % 2nd argument is what is printed
\newnumbered{definition}{Definition}
\newnumbered{hypothesis}{Hypothesis}
\newnumbered{remark}{Remark}
\newnumbered{note}{Note}
\newnumbered{observation}{Observation}
\newnumbered{problem}{Problem}
\newnumbered{question}{Question}
\newnumbered{algorithm}{Algorithm}
\newnumbered{example}{Example}
\newunnumbered{notation}{Notation} % This is usually unnumbered
\numberwithin{equation}{section}

\def\R{\mathrm{Re\,}}

\def\geq{\geqslant}
\def\leq{\leqslant}
\def\R{\mathbb{R}}
\def\T{\mathbb{T}}

\def\Z{\mathbb{Z}}
\def\N{\mathbb{N}}

\def\Q{\mathbb{Q}}
\def\cB{\mathcal{B}}

\def\cC{\mathcal{C}}

\newcommand{\bea}{\begin{eqnarray}}
  \newcommand{\eea}{\end{eqnarray}}
  \newcommand{\beab}{\begin{eqnarray*}}
  \newcommand{\eeab}{\end{eqnarray*}}

  \newcommand{\be}{\begin{equation}}
  \newcommand{\ee}{\end{equation}}

\newcommand{\ep}{\varepsilon}
\newcommand{\de}{\delta}
\newcommand{\g}{\gamma}
\newcommand{\al}{\alpha}
\newcommand{\la}{\lambda}
\title{Spectral disjointness of rescalings of some surface flows}
\author{P.\ Berk, A.\ Kanigowski}
\classno{37E35 (primary), 37A10 (secondary)}
\begin{document}
\baselineskip=14pt \maketitle

\begin{abstract} We study self-similarity problem for two classes of flows: 
\begin{enumerate}
\item special flows over circle rotations and under roof functions with symmetric logarithmic singularities 
\item special flows over interval exchange transformations and under roof functions which are of two types
\begin{itemize}
\item  piecewise constant with one additional discontinuity which is not a discontinuity of the IET;
\item piecewise linear over exchanged intervals with non-zero slope.
\end{itemize} 
\end{enumerate} 
We show that if
$\{T^{\alpha,f}_t\}_{t\in\R}$ is as in (1) then for a full measure set of rotations, and for every $K,L\in \N$, $K\neq L$, we have that $\{T^{\al,f}_{Kt}\}_{t\in\R}$ and $\{T^{\al,f}_{Lt}\}_{t\in\R}$ are spectrally disjoint. Similarly, if $\{T^f_t\}_{t\in\R}$ is as in (2), then for a full measure set of IET's, a.e. position of the additional discontinuity (of $f$, in piecewise constant case) and every $K,L\in \N$, $K\neq L$ the flows $\{T^f_{Kt}\}_{t\in\R}$ and $\{T^f_{Lt}\}_{t\in\R}$ are spectrally disjoint.
\end{abstract}

\tableofcontents

\section{Introduction}
The subject of this paper revolves around the {\em self-similarity} problem for measure preserving flows acting on probability standard Borel spaces. Recall that a flow $\{T_t\}_{t\in\R}$, where $T_t:(X,\cB,\mu)\to (X,\cB,\mu)$ is self-similar at {\em scale} $s\in \R$, if there exists an automorphism $S\in Aut(X,\cB,\mu)$ such that 
$$
T_t\circ S=S\circ T_{st},\;\;\text{ for every }\; t\in \R.
$$
More generally, one can try to describe the set of {\em joinings} between $\{T_t\}_{t\in\R}$ and $\{T_{st}\}_{t\in\R}$ (see Section \ref{sec:bac} for the definition of joinings). Self-similarities of flows were studied extensively by Danilenko and Rhyzhikov in \cite{DaRy}. Among many other results, they proved that the property of a flow being disjoint with all its {\em rescalings} is generic in the space of all measure preserving flows. Their result was  heavily influenced by the work of del Junco in \cite{dJun} on similar properties, but for automorphisms. We would also like to refer the reader to the papers \cite{FrLem3}, \cite{FrKLem3} and \cite{Kul} where the authors describe the set of self-similarities for various classes of flows. Some of the results in this paper base on the constructions given in \cite{BFr} where authors were focusing on relation of some special flows with their inverses, that is rescaling by $-1$. 

One more motivation for studying joinings rather then just self-similarity of different rescalings of a given flow $\{T_t\}_{t\in\R}$ has its justification in the so called Katai orthogonality criterion \cite{BSZ}, which has recently been proven to be an important tool for studying problems around Sarnak's conjecture on M{\"o}bius disjointness (see e.g.\ \cite{FKL}). From that point of view it is important to study disjointness (in the sense of Furstenberg) of $\{T_t\}_{t\in\R}$ and $\{T_{qt}\}$ for $q\in \Q$. A property that is stronger than disjointness is that of {\em spectral disjointness} (see Section \ref{sec:bac} for the definition) which will be the main focus of this paper. 

In this paper we are interested in spectral disjointness of rational rescalings of certain special flows coming from dynamics on surfaces. More precisely we will consider two types of special flows: 
\begin{enumerate}
\item special flows over irrational rotations on the circle and under the roof function $f$ with one {\em symmetric logarithmic singularity} (see Section \ref{sec:logsym}). We will denote such flows by $\{T_t^{\alpha,f}\}_{t\in\R}$;
\item special flows over interval exchange transformation  $T$ (IET for short) and under the roof function which is piecewise linear, we will denote such flows by $\{T^{f}_t\}_{t\in\R}$. 
\end{enumerate}
It has been shown in \cite{FrLem2} and more thoroughly in \cite{FrLem4}, that flows as in (1) are models of some area preserving flows on surfaces of Hamiltonian origin. Moreover,  flows as in (2) arise naturally  as special representations of some translation flows on surfaces (or their reparametrizations).

Notice that $\{T_t\}_{t\in\R}$ and $\displaystyle{\{T_{\frac{K}{L}t}\}_{t\in\R}}$ are spectrally disjoint if and only if $\{T_{Kt}\}_{t\in\R}$ and $\{T_{Lt}\}_{t\in\R}$ are spectrally disjoint. In the two theorems below, a.e. stands for almost every with respect to Lebesgue measure. Our two main results are:

\begin{theorem}\label{thm:1}
	Let $f:\mathbb T\to\R_+$ be given by
	\[
	f(x)=-C_f\log x-C_f\log (1-x)+g(x),
	\]
	where $C_f>0$ and $g\in C^3(\T)$, $g>0$. Then for a.e. $\alpha\in \T$ and every $K,L\in \N$, $K\neq L$, the flows $\{T^{\alpha,f}_{Kt}\}_{t\in\R}$ and $\{T^{\alpha,f}_{Lt}\}_{t\in\R}$ are spectrally disjoint.
\end{theorem}
In a recent paper \cite{KLU} the authors showed an analogous result (for disjointness in the sense of Furstenberg) for roof functions with {asymetric} logarithmic singularity. It is worth to mention however that their result relies heavily on the so called Ratner's property, while we present completely different approach.

We refer the reader to Section \ref{roie} for any yet undefined terms in the theorem below.
\begin{theorem}\label{Main}
Let $\mathfrak R$ be a Rauzy graph of irreducible permutations of alphabet $\mathcal A$ and consider $\mathfrak R\times \Lambda^{\mathcal A}$ with a product of counting and Lebesgue measure. Then there exists a set $\Upsilon\subset \mathfrak R\times \Lambda^{\mathcal A}$ of full measure such that for every interval exchange transformation $(\pi,\la)\in\Upsilon$ we have that
\begin{enumerate} 
	\item[(i)] for almost every $\beta\in[0,1)$ and piecewise constant positive function $f$ with jump in $\beta$ and optionally in the discontinuities of $T:=T_{\pi,\la}$,
	\item[(ii)] for a piecewise linear function with non-zero constant slope, linear over exchanged intervals, 
\end{enumerate}	
	the special flow $\{T^f_t\}_{t\in\R}$ satisfies 
\[
\{T^f_{Kt}\}_{t\in\R}\ \text{ and }\ \{T^f_{Lt}\}_{t\in\R}\ \text{ are spectrally disjoint for every distinct numbers }\ K,L\in\N.
\]
	\end{theorem}

The main tool for proving Theorems \ref{thm:1} and \ref{Main} is a criterion on spectral disjointness (see Section \ref{sec:dis}) which is based on similar criterion in \cite{ALR} and \cite{FrLem2}. We show how proving both Theorems falls down to differentiation of some limit probability measures (see Corollary \ref{cor:maintool}). In the case of Theorem \ref{thm:1} we use for this purpose the measure of properly picked rays, while in the case of { Theorem} \ref{Main} we provide a precise description of the limit measures.

The problem of self-similarity in both cases has been studied in literature. Namely in \cite{Kul} Kułaga-Przymus showed that for the special flows over IETs of bounded type and under roof functions with logarithmic singularities of symmetric type are not self-similar. On the other hand in \cite{FrLem3} Frączek and Lemańczyk showed that for almost every IET and a roof function with bounded variation the special flow is not isomorphic with all its rescalings (except -1).

\textbf{Outline of the paper:} The paper is organized as follows: in Section \ref{sec:bac} we introduce some basic definitions and notation. In Section \ref{sec:dis} we state the main criterion for spectral disjointness (Theorem \ref{krytspek}), which is then adjusted to deal with special flows (Theorem \ref{main} in Subsection \ref{sec:spec}). Finally in Corollary \ref{cor:maintool} 
we state sufficient condition for spectral disjointness of rescalings of a special flow. Then in Section \ref{sec:logsym} we prove Theorem \ref{thm:1} and in Section \ref{sec:pico} we prove Theorem \ref{Main}.

\textbf{Acknowledgements:} The research leading to these results was partially supported by the European Research Council under the European Union Seventh Framework Programme (FP/2007-2013) / ERC Grant Agreement n. 335989 and Narodowe Centrum Nauki Grant OPUS 14 2017/27/B/ST1/00078. The authors would like to thank K.\ Fr\k{a}czek and C.
 Ulcigrai for several discussions on the subject and the anonymous reviewer of this article for careful reading and detailed remarks.

\section{Basic definitions}\label{sec:bac}
In this section we recall some basic definitions and introduce notation that will be used throughout the paper. 
\subsection{Joinings, spectral theory}
Let $\mathcal T=\{T_t\}_{t\in\R}$ and $\mathcal S=\{S_t\}_{t\in\R}$ be two  measure preserving flows acting respectively on standard { Borel} probability spaces $(X,\mathcal B,\mu)$ and $(Y,\cC,\nu)$. A {\em joining} $\rho$ between $\mathcal T$ and $\mathcal S$ is a $\mathcal T\times \mathcal S$ invariant probability measure such that $\rho(B\times Y)=\mu(B)$ and $\rho(X\times C)=\nu(C)$, for all $B\in \cB$ and $C\in \cC$. 
The set of joinings is denoted by $J(\mathcal T,\mathcal S)$. Two flows are {\em disjoint} in the sense of Furstenberg if  $J(\mathcal T,\mathcal S)=\{\mu\otimes \nu\}$. We denote by $J_2(\mathcal T):=J(\mathcal T,\mathcal T)$ the set of {\em self-joinings} of $\mathcal T$. In particular, for $t\in \R$, an {\em off-diagonal} self joining $\mu_t\in J_2(\mathcal T)$ is defined, by
$$
\mu_t(A\times B)=\mu(A\cap T_{-t}B).
$$

We now recall basic definitions from spectral theory. Consider a natural group of unitary operators on $L^2(X,\mathcal B,\mu)$ given by
	$f\mapsto f\circ T_t$. These are called \emph{Koopman operators} and we denote them  by $\{U_t^{\mathcal T}\}_{t\in\R}$. A subspace $\R^{\mathcal T}_f=\overline{\operatorname{span}}\{f\circ T_t;\, t\in\R \}$ is a \emph{cyclic space} of element $f\in L^2(X,\mathcal B,\mu)$. Due to Bochner-Herglotz's theorem there exists a finite measure $\sigma_f$ on $\mathbb R$ such that
	\[
	\hat\sigma_f(t)=\langle f\circ T_{-t},f\rangle\ \text{ for every }\ t\in\R.
	\]
	We say that $\sigma_f$ is the \emph{spectral measure} of $f$. A\emph{ maximal spectral type }of $\mathcal T$ is a spectral measure $\sigma_{\mathcal T}$ such that for every $f\in L^2(X,\mathcal B,\mu)$ we have $\sigma_f\ll \sigma_{\mathcal T}$. We say that the flows $\mathcal T$ and $\mathcal S$ are \emph{spectrally disjoint} if $\sigma_{\mathcal T}\perp\sigma_{\mathcal S}$. Define the linear operator $V_t^\sigma$ on $L^2(\R,\mathcal B(\R),\sigma)$ such that
	\[
	V_t^\sigma(f)=e^{it(\cdot)}f\text{ for every }f\in L^2(\R,\mathcal B(\R),\sigma).
	\] 
	\begin{remark}\label{specisom} For every $f\in L^2(X,\mathcal B,\mu)$ the linear operators $U_t^{\mathcal T}$ on $\R^{\mathcal T}_f$ and $V_t^{\sigma_f}$ on $L^2(\R,\mathcal B(\R),\sigma_f)$ are isomorphic and the isomorphism is induced by the map $f\circ T_t\mapsto e^{it}$. We denote this induced isomorphism by $\phi_\mathcal T:\R^{\mathcal T}_f\to L^2(\R,\mathcal B(\R),\sigma_f)$.
		\end{remark}
	
	Recall, that spectral disjointness of $\mathcal T$ and $\mathcal S$ implies \emph{disjointness} in the sense of Furstenberg.	
	Let $P\in\mathcal P(\R)$\footnote{For given standard Borel space $(X,\mathcal B)$ we denote by $\mathcal P(X)$ the set of probability measures on $X$.}. We define an \emph{integral operator} $P(\mathcal T)$ in the following way:
	\[
	\text{for every }f,g\in L^2(X,\mathcal B,\mu),\text{ we have }\langle P(\mathcal T)f,g\rangle=\int_\R\langle f\circ T_{-t},g\rangle\,dP(t).
	\]

%	
%	It is worth to mention that there exists a direct relation between convergence of Koopman operators to an integral operator with the convergence of 2-off-diagonal joinings to an integral joining. Namely
%	\begin{equation}\label{opimi}
%		T_{t_n}\to P(\mathcal T)\ \text{ weakly if and only if }\ \mu_{t_n}\to\int_\R\mu_{-t}\,dP(t)\text{ in }J_2(\mathcal T).
%		\end{equation}

\subsection{Rotations and IET's}\label{roie}
\paragraph{\textbf{Irrational rotations}}
We recall some known facts on diophantine aproximation. For $\alpha\in \T\setminus \Q$ let $[0,a_1,a_2,\ldots]$ denote the continued fraction expansion of $\alpha$ and let $(p_n)_{n\in \N}$ and $(q_n)_{n\in \N}$ be  sequences of numerators and denominators for $\alpha$, i.e. $p_0=q_{-1}=1$, $p_{-1}=q_0=0$ and 
$$
p_{n+1}=a_{n+1}p_n+p_{n-1}\;\;\text{ and }\;\;q_{n+1}=a_{n+1}q_n+q_{n-1}.
$$
Then 
\be\label{eq:q_n}
\frac{1}{2q_nq_{n+1}}\leq \left\|\alpha-\frac{p_n}{q_n}\right\|\leq \frac{1}{q_nq_{n+1}},
\ee
where
\[
\|\beta\|:=\min\{\beta\!\mod 1,\,1-\beta\!\mod 1 \}\ \text{ for }\ \beta\in\R.
\]
For $n\in \N$ consider the partition $\mathcal{P}_n$ of $\T$ by the points $\{-i\alpha\}_{i=0}^{q_n-1}$. For $I\in \mathcal{P}_n$, we have 
\be\label{eq:lenn}
\frac{1}{q_{n}}-\frac{2}{q_{n+1}}\leq |I|\leq \frac{1}{q_n}+\frac{2}{q_{n+1}}.
\ee
\bigskip
\paragraph{\textbf{Interval exchange transformations.}}
Let now $\mathcal A$ be an alphabet of $d$ elements. We understand by permutation $\pi$ a pair of bijections $\pi_i:\mathcal A\to\{1,\ldots,d\}$ for $i=0,1$. { Let 
\[
\Lambda^{\mathcal A}:=\{\la\in\R^{\mathcal A}_{>0};\,\sum_{a\in\mathcal A}\la_a=~1~\}
\]
 be the $(d-1)$-dimensional unit simplex.} For every $a\in\mathcal A$ an $\la\in\Lambda^{\mathcal A}$ we denote
 \[
 I_a=\Big[\sum_{\{b;\,\pi_0(b)<\pi_0(a)\}}\la_b,\sum_{\{b;\,\pi_0(b)\le\pi_0(a)\}}\la_b\Big).
 \] 
  We say that $T_{\pi,\la}:[0,1)\to[0,1)$ is an \emph{interval exchange transformation}(or shortly \emph{IET}) if it rearranges intervals $\{I_a\}_{a\in\mathcal A}$ given by $\la$ according to the permutation $\pi$ by translations. Obviously this definition can be shifted to any interval $J\subset\R$ of arbitrary length via rescaling.
  
   Throughout this paper we only consider permutations $\pi=(\pi_0,\pi_1)$ which are \emph{irreducible} that is
\[
\pi_0\circ\pi_1^{-1}(\{1,\ldots,k\})=\{1,\ldots,k\}\quad\Rightarrow\quad k=d.
\] 
We use the standard notation $S_0^{\mathcal A}$ for the set of all irreducible permutations of alphabet~$\mathcal A$. We have the following theorem.
\begin{theorem}[\cite{Veech}]
	For almost every $(\pi,\la)\in S_0^{\mathcal A}\times\Lambda^{\mathcal A}$, with respect to product of counting and Lebesgue measure, $T_{\pi,\la}$ is ergodic.
	\end{theorem}

We consider the operator $R:S_0^{\mathcal A}\times \R_{>0}^{\mathcal A}\to S_0^{\mathcal A}\times \R_{>0}^{\mathcal A}$, such that $R(\pi,\la)=R(T_{\pi,\la})$ is %T a
the first return map of $T_{\pi,\la}$ to the interval $[0,|\la|-\min\{\la_{\pi_0^{-1}(d)},\la_{\pi_1^{-1}(d)}\})$. We denote $R^n(\pi,\la)=(\pi^n,\la^n)$ for any $n$ if this object is well defined. We also denote by $I^n$ the domain of the IET given by $(\pi^n,\la^n)$ and by $I^n_a$ for $a\in\mathcal A$ the elements of associated partition into exchanged intervals. Any minimal subset $\mathfrak R\subset S_0^{\mathcal A}$ invariant under induced action of $R$ is called a \emph{Rauzy graph}. We also consider a \emph{normalised Rauzy-Veech induction} $\tilde R:S_0^\mathcal A\times \Lambda^\mathcal A\to S_0^\mathcal A\times \Lambda^\mathcal A$ such that $\tilde R(\pi,\la)=(\pi^1,\frac{1}{|\la^1|}\la^1)$, where $|\cdot|$ denotes the sum of coefficients of a vector.

 We have the following series of important facts.
 \begin{theorem}[\cite{Ra}]\label{twpierost}
 	Any Rauzy graph of permutations of $d\ge 2$ elements contains at least one permutation $\pi$ such that
 	\[
 	\pi_1\circ\pi_0^{-1}(1)=d\quad \text{and}\quad\pi_1\circ\pi_0^{-1}(d)=1.
 	\]
 \end{theorem}
 \begin{theorem}[\cite{Ra}]\label{infi}
 	If $\pi\in S_0^{\mathcal A}$, then for almost every $\la\in\R_{> 0}^{\mathcal A}$ the Rauzy-Veech induction can be iterated infinitely many times on $(\pi,\la)$. Moreover $\lim_{n\to\infty}|\la^n|=0$.
 \end{theorem}
 \begin{theorem}[\cite{Mas} and \cite{Veech} independently]\label{powr}
 	For every Rauzy graph $\mathfrak R\subset S_0^{\mathcal A}$ there exists a measure $\rho$ on $\mathfrak R\times\Lambda^{\mathcal A}$ invariant under $\tilde R$ which is equivalent to the product of counting and Lebesgue measure. Then $\tilde R$ with the measure $\rho$ is ergodic and recurrent.
 	\end{theorem}
 	
 	Denote by $A^{(n)}(\pi,\la)=A^{(n)}=[A^{(n)}_{ab}]_{a,b\in\mathcal A}$ \emph{$n$-th Rauzy-Veech induction matrix} which is such that the coefficient $A^{(n)}_{ab}$ shows how many times the interval $I^n_b$ visits the interval $I_a$ under iterations of $T_{\pi,\la}$ before its first return to $I^n$.
 	
 	\begin{remark}\label{rozkladnawieze}
 		For almost every $(\pi,\la)\in \mathfrak R\times S_0^\mathcal A$ and for every $n\in\N$, the interval $I$ (which is the domain of $T_{\pi,\la}$) can be divided into~$d$ disjoint Rokhlin towers of the form $\{T^iI^n_b;\,0\leq i<s_b^n\}$, where $s_b^n:=\sum_{a\in\mathcal A}A_{ab}^n(\pi,\la)$
 		for $b\in\mathcal A$ (see\ Lemma~4.2 in
 		\cite{Viana}). More precisely,  $T^iI^n_b$ for $i=0,\ldots, s_b^n-1$ and $b\in\mathcal A$ are pairwise disjoint intervals filling up $I$. Moreover each such interval
 		$T^iI^n_b$ is included in some interval $I_a$ and $T^{s_b^n}I^n_b\subset I^n$. It follows that $T_{\pi,\la}$ acts on each interval $T^iI^n_b$ by translation and
 		$s_b^n$ is the time of first return of $I_b^{n}$ into $I^{n}$ via $T_{\pi,\la}$. Finally, the discontinuity points of $T_{\pi,\la}$ are leftpoints of some intervals $T^iI^n_b$.
 	\end{remark}
 	\begin{lemma}[\cite{Veech0}]\label{proposition}
 		Let $(\pi,\la)\in S_0^{\mathcal A}\times \R_{\ge 0}^{\mathcal A}$ be such that $R^n(\pi,\la)$ is well defined for every $n\in\N$. Then
 		\begin{enumerate}[(i)]
 			\item $A^{(n)}(\pi,\la)\la^n=\la$;
 			\item $A^{(n)}(\pi,\la)=A^{(1)}(\pi^0,\la^0)\cdot\ldots\cdot A^{(1)}(\pi^{n-1},\la^{n-1})$;
 			\item there exists $n\in \N$ such that the matrix $A^{(n)}(\pi,\la)$ is strictly positive
 			\item for almost every $(\pi,\la)\in S_0^{\mathcal A}\times \R_{\ge 0}^{\mathcal A}$  there exists $n\in \N$ such that the condition $(iii)$ is satisfied and $\pi^n=\pi$.
 		\end{enumerate}
 		\end{lemma}
Finally we also have the following result which is proven in \cite{BFr}:
\begin{lemma}\label{mnozmac}
Let $T_{\pi,\la}:[0,1)\to[0,1)$ be an interval exchange transformation such that $R(\pi,\la)$ is well defined.  Then for every $\lambda'\in\R_{>0}^\mathcal A$  we have
\[A^1(\pi,A^1(\pi,\la)\la')=A^1(\pi,\la)\text{ and }\pi^1(\pi,A^1(\pi,\la)\la')=\pi^1(\pi,\la).\]
\end{lemma}
For more information on interval exchange transformations we refer the reader to \cite{Viana}.

\subsection{Special flows}\label{sub:spec}
Let $T\in Aut(X,\cB,\mu)$ and $f\in L^1_+(X,\cB,\mu)$. We define the $\Z$-{\em cocycle}, by setting for $k\in\N\cup\{0\}$ 
$$
S_k(f)(x):=\sum_{i=0}^{k-1}f(T^ix),
$$
and, for $-k\in \N$, $S_k(f)(x):=-S_{-k}(f)(T^{k}x)$. Then the special flow $\mathcal T^f:=\{T_t^f\}_{t\in \R}$ is defined on $X^f:=\{(x,s)\;:\; x\in X, 0\leq s<f(x)\}$ by 
$$
T_t^f(x,s):=(T^nx,s+t-S_n(f)(x)),
$$
where $n\in \Z$ is unique such that $S_n(f)(x)\leq s+t<S_{n+1}(f)(x)$. The $\sigma$-algebra $\mathcal B^f$ on $X^f$ is $\cB\otimes \cB(\R)$ restricted to $X^f$ and $\{T_t^f\}_{t\in\R}$ preserves measure $\mu^f$ which is measure $\mu\otimes \lambda_\R$ restricted to $X^f$.

We will consider the case $T$ being an irrational rotation or an interval exchange transformation and the {\em roof function} $f$ satisfying some regularity conditions.

\section{Criterion on spectral disjointness.}\label{sec:dis}
In this section we will state a criterion on spectral disjointness of two measure-preserving flows. We have the following definition:
\begin{definition}\label{def:expdec}
We say that a measure $P\in\mathcal P(\R)$ has \emph{exponential decay} if there exist constants $c,b\in\R_{>0}$ such that
	\[
	P\big((-\infty,-t)\cup(t,\infty) \big)<ce^{-b t}\ \text{ for every }\ t\in\R_{>0}.
	\]
\end{definition}

	\begin{remark}\label{anal}
		If $P\in\mathcal P(\R)$ has exponential decay then the Fourier transform $\hat P(\cdot)$ is an analytic function. 
		\end{remark}
	
	We now present a criterion on two flows being spectrally disjoint. The proof of the following result is based on the proof of Remark 1 in \cite{LW}. 
	\begin{theorem}\label{krytspek}
		Assume that $\mathcal T=\{T_t\}_{t\in\R}$ and $\mathcal S=\{S_t\}_{t\in\R}$ are weakly mixing flows on probability spaces $(X,\mathcal B,\mu)$ and $(Y,\mathcal C,\nu)$ respectively. Suppose that there exists a sequence $\{t_n\}_{n\in\N}$ increasing to infinity such that  
		\begin{equation}\label{zbie}
			T_{t_n}\to P(\mathcal T)\quad \text{and}\quad S_{t_n}\to Q(\mathcal S)\quad \text{weakly},
			\end{equation}
		where $P,Q$ are probability measures with exponential decay on $\R$. If $P\neq Q$ then $\mathcal T$ and $\mathcal S$ are spectrally disjoint.
		\end{theorem}
		\begin{proof}
	Suppose by contradiction that $\mathcal T$ and $\mathcal S$ are not spectrally disjoint. Then there exists a measure $0\neq\rho\in \mathcal P(\R)$ such that $\rho\ll\sigma_{\mathcal T}$ and $\rho\ll\sigma_{\mathcal S}$. Since the flows $\mathcal T$ and $\mathcal S$ are weakly mixing, the measure $\rho$ is { non-atomic}.
	Then there are functions $f\in L^2(X,\mathcal B,\mu)$ and $g\in L^2(Y,\mathcal C,\nu)$ such that the operators $U_t^{\mathcal T}$ on $\R_{f}^{\mathcal T}$ and $U_t^{\mathcal S}$ on $\R_{g}^{\mathcal S}$ are isomorphic to an operator $V^\rho_t$ on $L^2(\R,\mathcal B(\R),\rho)$ for every $t\in\R$ via the 
%	Let $f:X\to\mathbb C$ and $g:Y\to\mathbb C$ be generators of spaces $\R_{\rho}^{\mathcal T}$ and $\R_{\rho}^{\mathcal S}$ respectively, that is
%	 \[
%	 \overline{\operatorname{span}}\{f\circ T_{t};\ t\in\R\}=\R_{\rho}^{\mathcal T}\ \text{ and }\ \overline{\operatorname{span}}\{g\circ S_{t};\ t\in\R\}=\R_{\rho}^{\mathcal S}.
%	 \]
%	 Moreover,
%	 \[
%	 \langle f\circ T_t,f\rangle = \int_\R e^{it}\,d\rho(t)\ \text{ and }\ \langle g\circ S_t,g\rangle = \int_\R e^{it}\,d\rho(t)\ \text{ for every }t\in\R
%	 \]
%	 and the map $f\circ T_t\mapsto e^{it}$ ($g\circ S_t\mapsto e^{it}$) defines
	  isomorphisms $\phi_\mathcal T$ and $\phi_{\mathcal S}$ respectively (see Remark \ref{specisom}).
	  % between $\R_{\rho}^{\mathcal T}$ ($\R_{\rho}^{\mathcal S}$) and $L^2(\R,\mathcal B(\R),\rho)$ which commutes with the action of $\mathcal T$ $(\mathcal S)$ and multiplication by $e^{it}$.
	 
	 Let $t_n$ be the sequence  given by the assumptions of the theorem.
	 Let also $F,G\in L^2(\R,\mathcal B(\R),\rho)$ be arbitrary. Then
	 	 \[
	 	 \begin{split}
	 	 \langle V_{t_n}^\rho F,G \rangle_{\rho}&=\langle\phi_{\mathcal T}^{-1}F\circ T_{t_n},\phi_{\mathcal T}^{-1}G \rangle_\mu	 \to \langle P(\mathcal T)(\phi_{\mathcal T}^{-1}F),\phi_{\mathcal T}^{-1}G\rangle_\mu\\
	 	 &= \int_\R \langle\phi_{\mathcal T}^{-1}F\circ T_{-t},\phi_{\mathcal T}^{-1}G\rangle_\mu\,dP(t)=\int_\R\int_\R e^{its}F(s)\overline{G(s)}\,d\rho(s) dP(t)\\
	 	 &=\int_\R \hat P(s)F(s)\overline{G(s)}\,d\rho(s)
	 	 =\langle\hat P(\cdot)F,G \rangle_\rho.
	 	 \end{split}
	 	 \]
%	 \[
%	 \begin{split}
%	 \langle V_{t_n}^\rho F,G \rangle_{\rho}&=\int_\R e^{it_ns}F(s)\overline{ G(s)}\,d\rho(s)=\int_X\phi_{\mathcal T}^{-1}\big(e^{it_n(\cdot)}F\big)\cdot\overline{\phi_{\mathcal T}^{-1}G}\,d\mu\\
%	 &=\int_X \phi_{\mathcal T}^{-1}F\circ T_{t_n}\cdot\overline{\phi_{\mathcal T}^{-1}G}\,d\mu(x)=\langle\phi_{\mathcal T}^{-1}F\circ T_{t_n},\phi_{\mathcal T}^{-1}G \rangle_\mu\\
%	 &\to \langle P(\mathcal T)(\phi_{\mathcal T}^{-1}F),\phi_{\mathcal T}^{-1}G\rangle_\mu= \int_\R \langle\phi_{\mathcal T}^{-1}F\circ T_{-t},\phi_{\mathcal T}^{-1}G\rangle_\mu\,dP(t)\\
%	 &=\int_\R\int_\R e^{its}F(s)\overline{G(s)}\,d\rho(s) dP(t)=\int_\R \hat P(s)F(s)\overline{G(s)}\,d\rho(s)\\
%	 &=\langle\hat P(\cdot)F,G \rangle_\rho.
%	 \end{split}
%	 \]
	 On the other hand, by replacing $\mathcal T$ with $\mathcal S$ we get 
	 \[
	 \langle V_{t_n}^\rho F,G \rangle_\rho\to\langle \hat Q(\cdot)F,G\rangle_\rho.
	 \]
	 Hence the operators of multiplying by $\hat P$ and $\hat Q$ are identical on $L^2(\R,\mathcal B(\R),\rho)$. Thus we obtain that $\hat P=\hat Q$ $\rho$-a.e. Since $\rho$ is a { non-atomic} measure, by Remark \ref{anal} we get that $\hat P=\hat Q$ everywhere on $\R$. This implies that measures $P$ and $Q$ are equal and yields a contradiction with the assumption of the theorem. 
	\end{proof}
%	
%The following criterion is a direct consequence of Theorem \ref{krytspek}.
%% (see also \eqref{opimi}).
%	\begin{corollary}\label{main1}
%		Let $\mathcal T=\{T_t\}_{t\in\R}$ and $\mathcal S=\{S_t\}_{t\in\R}$ be  weakly mixing flows on $(X,\mathcal B,\mu)$ and $(Y,\mathcal C,\nu)$ respectively.
%		Assume that for some real sequence $(a_n)_{n\in\N}$ we have
%		\[
%		T_{a_n}\to P(\mathcal T)\quad \text{and}\quad 
%		S_{a_n}\to Q(\mathcal S)
%		\]
%		for some measures $P,Q\in\mathcal P(\R)$ with exponential decay. Assume moreover, that $P\neq Q$.
%		Then $\mathcal T$ and $\mathcal S$ are spectrally disjoint.
%	\end{corollary}

\subsection{Spectral disjointness for special flows}\label{sec:spec}

The following result was proven in \cite{FrKuLem} and in more general version in \cite{BFr} and gives the description of the aforementioned limits under some natural assumptions. Recall that for a dynamical system $(X,\mathcal B,\mu, T)$ we say that $\{q_n\}_{n\in\N}$ is \emph{a rigidity sequence along a sequence of subsets} $\{W_n\}_{n\in\N}$ if for every measurable $A\subset \mathcal B$ we have 
	\[
	\lim_{n\to\infty}\mu((A\triangle T^{q_n}A)\cap W_n)=0.
	\]
The following auxiliary result will be useful in proving Theorem \ref{Main}. It is  proven in \cite{BFr}.
\begin{lemma}[see \cite{BFr}]\label{metr}
	Suppose that $(X,\mathcal B,\mu)$ is endowed with a metric $d$
	generating the $\sigma$-algebra~$\mathcal B$. If $\sup_{x\in
		W_n}d(T^{q_n}x,x)\to 0$, then $\{q_n\}_{n\in\N}$ is a rigidity
	sequence for $T$ along $\{W_n\}_{n\in\N}$ .
\end{lemma}

	\begin{theorem}[see \cite{FrLem5}]\label{main}
		Suppose that there exists a sequence $\{W_n\}_{n\in\N}$ of measurable subsets of $X$, increasing sequence $\{q_n\}_{n\in\N}$ of natural numbers and real sequence $\{a_n\}_{n\in\N}$,such that following conditions are satisfied:
		\begin{equation}\label{wl1}
			\mu(W_n)\to 1,
		\end{equation}
%		\begin{equation}\label{wl2}
%			\mu(W_n\triangle T^{-1}W_n)\to 0,
%		\end{equation}
		\begin{equation}\label{wl3}
			\{q_n\}_{n\in\N}\text{ is a rigidity sequence for $T$ along }\{W_n\}_{n\in\N},
		\end{equation}
		%\begin{equation}\label{wl4}
		%\{q_n'\}_{n\in\N}\text{ is a rigidity sequence for $T$ along }\{W_n\}_{n\in\N},
		%\end{equation}
		\begin{equation}\label{wl5}
			\Big\{\int_{W_n}|S_{q_n}(f)(x)-a_n|^2d\mu(x)\Big\}_{n\in\N}\text{ is bounded},		\end{equation}
%		\begin{equation}\label{wl6}
%			\Big\{\int_{W_n}|f_n'(x)|^2d\mu(x)\Big\}_{n\in\N}\text{ is bounded for }f_n'=f^{(2q_n)}-2a_n,
%		\end{equation}
		\begin{equation}\label{wl7}
			(S_{q_n}(f)(x)-a_n)_*(\mu_{W_n})\to P\text{ weakly in }\mathcal{P}(\R)\footnote{Throughout this paper for any measure $\mu$ we denote by $\mu_A$ the conditional measure on the set $A$, while by $\mu|_A$ we denote the restriction of the measure $\mu$ to the set $A$ }.
		\end{equation}
		Then
		\[
		T_{a_n}^f\to P(\mathcal T^f),
		\]
		up to a subsequence.
	\end{theorem}
We will now state a crucial corollary, which is the most important tool in proving spectral disjointness of different rescalings of a special flow $\{T_t^f\}_{t\in\R}$. For a measure $P\in \mathcal{P}(\R) $ and $w\in \R$ we denote by $Res_w(P)\in \mathcal{P}(\R)$ the measure given by $[Res_w(P)](A)=P(w\cdot A)$. Note that if $P$ is an absolutely continuous measure with density $x\mapsto f(x)$ then $Res_w(P)$ is also an absolutely continuous measure with density $x\mapsto wf(w\cdot x)$.
{\begin{corollary}\label{cor:maintool} Let $(T_t^f)$ be a weakly mixing special flow on probability standard Borel space $(X^f,\mathcal B^f,\mu^f)$ and fix $K,L\in \N$, $K\neq L$. Let $\{W_n\}_{n\in \N}$, $\{q_n\}_{n\in \N}$ and 
$\{a_n\}_{n\in \N}$
be such that \eqref{wl1}-\eqref{wl5} hold for $\{wq_n\}_{n\in\N}$ and $\{wa_n\}_{n\in\N}$, where $w\in\{K,L\}$. Assume moreover that
\be\label{eq:convme}
\begin{split}
&(S_{Kq_n}(f)(x)-Ka_n)_*(\mu_{W_n})\to P_K\text{ weakly in }\mathcal{P}(\R)\quad\text{and}\\
&(S_{Lq_n}(f)(x)-La_n)_*(\mu_{W_n})\to P_L\text{ weakly in }\mathcal{P}(\R),
\end{split}
\ee
where both $P_K$ and $P_L$ have exponential decay. If $Res_K(P_K)\neq Res_L(P_L)$ then $\mathcal T_{K}^f=\{T_{Kt}^f\}_{t\in\R}$ and $\mathcal T_L^f=\{T_{Lt}^f\}_{t\in\R}$ are spectrally disjoint.
\end{corollary}}
\begin{proof}
%Note first that $\{Kq_n\}_{n\in\N}$ and $\{Lq_n\}_{n\in\N}$ are rigidity sequences along $\{W_n\}_{n\in\N}$. Indeed, by $T$-invariance of$\mu$ for every set $A\in\mathcal B$ we have
%\[
%\mu((A\triangle T^{Kq_n})\cap W_n)\le\sum_{i=0}^{K-1}\mu((T^{iq_n}A\triangle T^{(i+1)q_n}A)\cap W_n)=\sum_{i=0}^{K-1}\mu((A\triangle T^{q_n}A)\cap T^{-iq_n}W_n).
%\]
%In view of \eqref{wl1} and \eqref{wl3} for every $i=0,\ldots,K-1$ we have
%\[
%\begin{split}
%\mu((A\triangle T^{q_n}A)\cap T^{-iq_n}W_n)=
%\end{split}
%\]

 In view of Theorem \ref{main} we have the following convergences
\[
T_{Ka_n}^f\to P_K(\mathcal T)=Res_K(P_K)(\mathcal T_K^f)\quad\text{and}\quad T_{La_n}^f\to P_L(\mathcal T)=Res_L(P_L)(\mathcal T_L^f).
\]
%\[
%\mu_{Ka_n}^f\to\int_{\R}\mu_{-t}\,dP_K(t)=\int_{\R}\mu_{-Kt}\,dRes_K(P_K)(t)
%\]
%and
%\[
%\mu_{La_n}^f\to\int_{\R}\mu_{-t}\,dP_L(t)=\int_{\R}\mu_{-Lt}\,dRes_L(P_L)(t).
%\]
Since $Res_K(P_K)\neq Res_L(P_L)$, Theorem \ref{krytspek} yields the desired result.
\end{proof}
\begin{remark} Note that assuming weak mixing in the above corollary is natural. Indeed, if a flow $\{T_t\}_{t\in\R}$ has a non-zero eigenvalue $a\in\R$ then $KLa$ is a common eigenvalue for $\{T_{Kt}\}_{t\in\R}$ and $\{T_{Lt}\}_{t\in\R}$ for every $K,L\in\N$. Thus the natural powers of a non-trivial ergodic flow which is not weakly mixing are never pairwise spectrally disjoint.
	\end{remark}
By definition of weak convergence of measures we also have the following remark:
\begin{remark}\label{rem:messet} It follows that if for $K\in \N$ and $P_K$ is as in Corollary \ref{cor:maintool}, then for every $b\in\R$ modulo a countable set, we have
$$
Res_K(P_K)((-\infty,b]\cup[b,\infty))=\lim_{n\to +\infty}\mu\left(\{x:\in \T\;:\; |S_{Kq_n}(f)(x)-Ka_n|\geq Kb\}\right).
$$
\end{remark}

\section{Flows over rotations and logarithmic singularity}\label{sec:logsym}
In this section we will apply the criterion from Section \ref{sec:spec} for special flows over irrational rotations and under roof function with one symmetric logarithmic singularity (which we assume WLOG is placed at zero). { More precisely, the roof function, $f:\T\to\R_+$, is given by:
\be\label{eq:logsym}
f(x)=-\log x-\log (1-x)+g(x),
\ee
where $g\in C^3(\T)$, $g>0$. All the reasoning which follows works also for any function of the above form multiplied by an arbitrary constant $C_f>0$.} The results of this section are largely based on the results presented by C.\ Ulcigrai in \cite{CU} for general IETs.
\bigskip
\subsection{Diophantine condition on the base rotation:}
Recall that for $\alpha\in \T$, $\{a_n\}_{n\in \N}$ denotes the continued fraction sequence of $\alpha$. For $K,L
\in \N$ we say that $\alpha\in \cC_{K,L}$ if there exists an increasing sequence $(n_k)_{k\in \N}$ and a constant $c>0$ such that 
\be\label{eq:diophcon}
a_{n_k+1-n}\leq e^{cn}\;\;\; \text{  for every }\;\; 1\leq n\leq n_k\;\;\text{and} \;\;\;200(K^2+L^2)>a_{n_k+1}>100(K^2+L^2).
\ee
Let $\cC:=\bigcap_{K,L\in \N} \cC_{K,L}$.
%\begin{remark}\label{rem:diophcon} The first part of the above diophantine condition was introduced for general IET's by C.\ Ulcigrai in \cite{CU}. As in \cite{CU}, this condition is crucial in fine estimates of Birkhoff sums of $f'$. However, in our case we need such estimates for ALL $x\in \T$, whereas in \cite{CU}, the additional {\em balance} condition for IET's  allows to control only a portion of space.
%\end{remark}
%We have the following lemma which is a  consequence of Lemma 2.2 in \cite{CU}:
\begin{lemma}\label{lem:fms}
With $\cC$ defined as above we have $\lambda(\cC)=1$.
\end{lemma}
\begin{proof}
Note that it is enough to show that for given $K,L\in\N$ we have $\la(\cC_{K,L})=1$. To prove this result we use the notion of natural extension $\hat G$ of Gauss map. That is we consider { the} map $\hat G$ on $\mathcal D:=(0,1)\setminus \Q\times(0,1)\setminus \Q$ given by the formula
\[
\hat G([0;a_1,a_2,\ldots],[0;a_0,a_{-1},\ldots])
=([0;a_2,a_3,\ldots],[0;a_1,a_0,a_{-1}\ldots]).
\]
This transformation preserves the measure $\la_{\hat G}$ given by the density $\frac{1}{\ln 2(1+\al^+\al^-)^2}$, where $\al^+=[0;a_1,a_2,\ldots]$ and $\al^-=[0;a_0,a_{-1},\ldots]$. In particular $\la_{\hat G}$ projects on each coordinate as a regular Gauss measure $\la_G=\frac{1}{\ln 2(1+x)}\,dx$. Moreover $\hat G$ with $\la_{\hat G}$ is an ergodic system (see i.e. \cite{UlSin}). 

Let
\[
\mathcal D_c:=\{([0;a_1,a_2,\ldots],[0;a_0,a_{-1},\ldots]);\ \forall_{n\ge 0}\,a_{-n}\le e^{c(n+1)}\ \text{ and }\ 100(K^2+L^2)<a_1<200(K^2+L^2)\}.
\]
Denote
 \[
 b:=\la_G\left(\{\al\in\mathbb T; a_{1}\ge200(K^2+L^2)\vee a_{1}\le100(K^2+L^2)\} \right)<1.
 \]	
Then 
\[
\la_{\hat G}(\mathcal D\setminus\mathcal D_c)\le b+\sum_{i=1}^{\infty}\ln(1+e^{-ci})\le b+\frac{e^{-c}}{1-e^{-c}}.
\]
By taking $c$ sufficiently large, we get that $\la_{\hat G}(\mathcal D_c)>0$. By the ergodicity of $\hat G$ it implies that for almost every $(\al,\beta)\in\mathcal D$ there exists an increasing sequence $\{n_k\}_{k\in\N}$ such that $\hat G^{n_k}(\al,\beta)\in \mathcal D_c$. This implies that $\al\in \cC$ with constant $c$ and thus concludes the proof.

\end{proof}

\subsection{Estimates on Birkhoff sums of  derivatives}
In this section we will estimate Birkhoff sums of the roof function (and { its} derivaties). One of the main tools is the the Denjoy-Koksma inequality.  For $n\in \N$ let 
\be\label{eq:dist111}
x^-_{n,min}:=\min_{0\leq j<q_n}x+j\alpha\;\;\text { and } \;\;
x^+_{n,min}:=\min_{0\leq j<q_n}1-(x+j\alpha)
\ee
considered on $[0,1)\subset \R$.
Let moreover $x_{n,min}:=\min(x^+_{n,min},x^-_{n,min})$. 

\textbf{General Denjoy-Koksma type estimates.}
The two lemmas below hold for every irrational $\alpha$.
\begin{lemma}\label{lem:DK} There exists a constant $C'>0$ such that for every $n\in \N$
$$
|S_{q_n}(f')(x)|<C'\left(q_n+\frac{1}{x_{n,min}}\right).
$$
\end{lemma}
\begin{proof} Notice that by \eqref{eq:logsym}, we have $f'=\frac{1}{1-x}-\frac{1}{x}+g'$, where $g'\in C^2(\T)$ and $\int_{\T}g'd\la=0$. By Denjoy-Koksma inequality for every $x\in \T$ and every $n\in \N$ 
\be\label{eq:DK1}
|S_{q_n}(g')(x)|<C'',
\ee
for some $C''>0$. Hence we will only deal with  $h(x):=\frac{1}{1-x}-\frac{1}{x}$.
Let $\bar{f}(x):=\chi_{[\frac{1}{4q_n},1-\frac{1}{4q_n}]}(x)h(x)$ for $x\in\T$. Since $\|q_{n-1}\alpha\|\geq \frac{1}{2q_n}$ (see \eqref{eq:q_n}), it follows that  
$$
\{x,x+\alpha,\ldots,x+(q_{n}-1)\alpha\}\cap \left[-\frac{1}{4q_{n}},\frac{1}{4q_{n}}\right]\subset \{x+j\alpha\},
$$
where $0\leq j< q_n$ is such that $\|x+j\alpha\|=x_{n,min}$ (notice that the set above might be empty). This implies that 
\be\label{eq:esDK}
|S_{q_n}(\bar{f})(x)-S_{q_n}(h)(x)|\leq \frac{1}{x_{n,min}}.
\ee
Moreover, the function $\bar{f}$ is of bounded variation, and hence by Denjoy-Koksma inequality, we have 
$$
|S_{q_n}(\bar{f})(x)-q_n\int_\T\bar{f}d\lambda|<{\rm Var}(\bar{f}).
$$
By the definition of $\bar{f}$ and $h$ it follows that $q_n\int_\T\bar{f}(x)d\lambda=q_n\int_{\frac{1}{4q_n}}^{1-\frac{1}{4q_n}}h(x)d \lambda=0$ and ${\rm Var}(\bar{f})\leq 16 q_n$. This and \eqref{eq:esDK} finish the proof.
\end{proof}

The following lemma estimates Birkhoff sums for higher order derivatives\footnote{For a  function $h:\T\to \R$ and $r\geq 1$, we denote by $h^{(r)}$ the $r$-th derivative of $h$.} of symmetric log:

\begin{lemma}\label{lem:DKlog} Let $h(x)=-\log x-\log (1-x)$. Then for every $r\geq 1$, $n\in \N$ and $x\in \T$, we have 
$$
\left|S_{q_n}(h^{(r)})(x)- (r-1)!\left[(-1)^{r}\left(\frac{1}{x^-_{n,min}}\right)^r+\left(\frac{1}{x^+_{n,min}}\right)^r\right]\right|<100(r-1)!\|q_{n-1}\alpha\|^{-r}.
$$
\end{lemma}
\begin{proof}Let $h_1(x)=-\log x$ and $h_2(x)=-\log(1-x)$. Then 
$h_1^{(r)}(x)=(-1)^r(r-1)!x^{-r}$ and $h_2^{(r)}(x)=(r-1)!(1-x)^{-r}$.
Assume first that $r\geq 2$. We will estimate Birkhoff sums of $h_1^{(r)}$, the estimates for $h_2^{(r)}$ are analogous. Let 
$0<x_{n,min}^-=x+j_1\alpha<x+j_2\alpha<\ldots<x+j_{q_n}\alpha<1$
be all the points in the $q_n$ orbit of $x$ in increasing order. Then 
$x+j_2\alpha\geq \|q_{n-1}\alpha\|$ (since $|j_2-j_1|<q_{n}$). Let
$\tilde{h}_{1,r}(z):=h_1^{(r)}(z)\chi_{[x+j_2\alpha,1]}(z)$. Then \be\label{eq:imt}
S_{q_n}(h_1^{(r)})(x)-h_1^{(r)}(x_{n,min}^-)=S_{q_n}(\tilde{h}_{1,r})(x)
\ee
and $\tilde{h}_{1,r}$ has bounded variation. So by Denjoy-Koksma inequality for $\tilde{h}_{1,r}$, we get 
\be\label{eq:DKfg}
|S_{q_n}(\tilde{h}_{1,r})(x)-q_n\int_\T \tilde{h}_{1,r}(z) d\lambda|<{\rm Var}( \tilde{h}_{1,r}).
\ee
By monotonicity of $h_1$ it follows that 
\begin{multline*}
\left|\int_\T \tilde{h}_{1,r}\,d\lambda\right|=\left|\int_{x+j_2\alpha}^1 h_1^{(r)}\, d\lambda\right|\leq  |h_1^{(r-1)}(1)|+ |h_1^{(r-1)}(x+j_2\alpha)|\leq\\
 (r-2)!+|h_1^{(r-1)}(\|q_{n-1}\alpha\|)|\leq 2(r-1)!\|q_{n-1}\alpha\|^{-r+1}.
\end{multline*}
Moreover, by the definition of $\tilde{h}_{1,r}$ it follows that  ${\rm Var}( \tilde{h}_{1,r})\leq 4|h_1^{(r)}(x+j_2\alpha)|\leq 4(r-1)!\|q_{n-1}\alpha\|^{-r}$.
Since $\frac{2}{q_n}>\|q_{n-1}\alpha\|\geq \frac{1}{2q_n}$, from \eqref{eq:DKfg} and \eqref{eq:imt}, we get 
$$
|S_{q_n}(h_1^{(r)})(x)-(-1)^{r}(r-1)!\left(\frac{1}{x^-_{n,min}}\right)|\leq 8(r-1)!\|q_{n-1}\alpha\|^{-r}.
$$

Analogously we show that 
$$
|S_{q_n}(h_2^{(r)})(x)-(r-1)!\left(\frac{1}{x^+_{n,min}}\right)|\leq 8(r-1)!\|q_{n-1}\alpha\|^{-r}.
$$
Since $h^{(r)}=h_1^{(r)}+h_2^{(r)}$, the two above inequalities finish the proof for $r\geq 2$. 

The proof in case $r=1$ follows similar steps, however one defines
$\tilde{h}_{1}(z):=h'_1(z)\chi_{[\|q_{n-1}\alpha\|,1]}(z)$. Then \eqref{eq:imt} becomes 
$$
|S_{q_n}(h'_1)(x)-h'_1(x_{n,min}^-)-S_{q_n}(\tilde{h}_{1})(x)|\leq \frac{1}{\|q_{n-1}\al\|}.
$$
Indeed, if $x_{n,min}^-\in[0,\|q_{n-1}\al\|]$ then the LHS of the above inequality is equal to $0$. Otherwise, it is equal to $h_1'(x_{n,min}^-)=-\frac{1}{x_{n,min}^-}\geq \frac{1}{\|q_{n-1}\al \|}$.
Analogously
$$
|S_{q_n}(h'_2)(x)-h'_2(x_{n,min}^+)-S_{q_n}(\tilde{h}_{2})(x)|\leq \frac{1}{\|q_{n-1}\al\|}.
$$
Next, by Denjoy-Koksma inequality and the definition of $\tilde{h}_1$, we have
$$
|S_{q_n}(\tilde{h}_{1})(x)-q_n\int_{\|q_{n-1}\alpha\|}^1h'_1(z)\, d\lambda|<{\rm Var}( \tilde{h}_1),
$$
and ${\rm Var}( \tilde{h}_1)\leq 4\|q_{n-1}\alpha\|^{-1}$. Analogously, we get 
$$
|S_{q_n}(\tilde{h}_{2})(x)-q_n\int_0^{1-\|q_{n-1}\alpha\|}h'_2(z)\, d\lambda|<{\rm Var}( \tilde{h}_2),
$$
with ${\rm Var}( \tilde{h}_2)\leq 4\|q_{n-1}\alpha\|^{-1}$. Putting all this together and using the fact that $-\int_{\|q_{n-1}\alpha\|}^1h'_1\, d\lambda=\int_0^{1-\|q_{n-1}\alpha\|}h'_2\,d\lambda$, we get 
\begin{multline*}
\left|S_{q_n}(h')(x)-\left[-\left(\frac{1}{x^-_{n,min}}\right)+\left(\frac{1}{x^+_{n,min}}\right)\right]\right|\leq \\
{\rm Var}( \tilde{h}_1)+{\rm Var}( \tilde{h}_2)+2\|q_{n-1}\al\|^{-1}\leq 10\|q_{n-1}\alpha\|^{-1}.
\end{multline*}
This finishes the proof for $r=1$.
\end{proof}

\textbf{Estimates relying on the diophantine condition.}
In this paragraph we assume that $\alpha \in \cC$ and that $(n_k)_{k\in \N}$ is a sequence satisfying \eqref{eq:diophcon} for $\alpha$. The estimates and elements of proofs in this section should be compared with  Section 4 in \cite{CU} (Proposition 4.2. in particular). 
%The difference in our case is that we need to control ALL points along one fixed sequence $(q_{n_k})_{k\in \N}$.

\begin{lemma}\label{lem:Best} For every $c''>0$ there exists a constant $C''>0$  such that for every $k\in \N$ and every $x\in \T$ satisfying 
$$
\{x,x+\alpha,\ldots,x+(q_{n_k}-1)\alpha\}\cap \left[\frac{-c''}{q_{n_k}},\frac{c''}{q_{n_k}}\right]=\emptyset,
$$
we have that for every $\ell\in [0,q_{n_k}]\cap \N$ it holds that
$$
|S_\ell(f')(x)|\leq C''q_{n_k}.
$$
\end{lemma}
\begin{proof}  We use Ostrovski expansion to write $\ell=\sum_{i=1}^{n_k-1}b_iq_i$, for $0\leq b_i\leq a_i$. Let $x\in\mathbb T$ satisfy the assumptions of the lemma. Let $x(0,s):=x+sq_{n_{k-1}}\alpha$  for $0\le s<b_{n_k-1}$ and let $x(i,s):=x+\left(\sum_{j=0}^{i-1} b_{n_k-1-j}q_{n_k-1-j}+ s q_{n_k-1-i}\right)\alpha$  for $1\leq i\leq n_k-1$ and $0\leq s< b_{n_k-i-1}$. We may assume that the point of the orbit which minimizes the distance from $0$ is one of the first $b_{n_k-1}q_{n_k-1}$ iterations of $x$. Then we split into orbits of length $q_i$:
$$
S_{\ell}(f')(x)=\sum_{i=0}^{n_k-1}\left(\sum_{s=0}^{b_{n_k-1-i}-1}S_{q_{n_k-1-i}}(f')(x(i,s))\right)
$$
If the point of the orbit which minimizes the distance from $0$ is one of the last $b_{n_k-1}q_{n_k-1}$ iterations of $x$, then we proceed analogously by defining $x(0,s):=T^{\ell}x-sq_{n_{k-1}}\alpha$  for $0\le s<b_{n_k-1}$ and\\ $x(i,s):=T^{\ell}x-\left(\sum_{j=0}^{i-1} b_{n_k-1-j}q_{n_k-1-j}+ s q_{n_k-1-i}\right)\alpha$  for $1\leq i\leq n_k-1$ and $0\leq s< b_{n_k-i-1}$ and considering representation 
$$
S_{\ell}(f')(x)=\sum_{i=0}^{n_k-1}\left(-\sum_{s=0}^{b_{n_k-1-i}}S_{-q_{n_k-1-i}}(f')(x(i,s))\right)
$$
The rest of the calculations is symmetric.

By Lemma \ref{lem:DK}, we have 
\be\label{eq:splt}
|S_{q_{n_k-1-i}}(f')(x(i,s))|\leq C'q_{n_k-1-i}+C' \left[(x(i,s))_{n_k-1-i,min}\right]^{-1}.
\ee
Notice that we can assume that WLOG we have that for every $0\le i \le n_{k}-1$ and $0\le s<b_{n_k-i-1}$ it holds that $(x(i,s))_{n_k-1-i,min}=(x(i,s))_{n_k-1-i,min}^+$. Indeed, we have that $\left[(x(i,s))_{n_k-1-i,min}\right]^{-1}<\left[(x(i,s))_{n_k-1-i,min}^+\right]^{-1}+\left[(x(i,s))_{n_k-1-i,min}^-\right]^{-1}$ and the estimate for the expressions of the second kind is analogous\footnote{Notice that by \eqref{eq:dist111} for $y\in \T$ and $n\in \N$, the number $y_{n,min}$ represents the distance of the closest point of the $q_n$ orbit of $y$ to $0$ (on $\T$). In what follows, the notation $y_{n,min}$ will also denote the point in the orbit which minimizes the distance.}.

Denote $\tilde{x}_{i}:=\min_{0\leq s\leq b_{n_k-1-i}}(x(i,s))_{n_k-1-i,min}$. Recall that we assumed that $\tilde{x}_0$ is the point which minimizes the distance of the whole orbit from $0$. Notice that for $0\leq s<b_{n_k-1-i}$,  the points $(x(i,s))_{n_k-1-i,min}$ together with $\tilde x_{i+1}$ (we put $\tilde x_{n_k}=\infty$) are distinct
and they belong to the orbit of length $q_{n_k-i}$ of the point $x(i,0)$. So for $0\leq s,s'< b_{n_k-1-i}$, $s\neq s'$, we have 
$$
\|(x(i,s))_{n_k-1-i,min}-(x(i,s'))_{n_k-1-i,min}\|\geq \frac{1}{2q_{n_k-i}}.
$$
and
\[
\|(x(i,s))_{n_k-1-i,min}-\tilde{x}_{i+1}\|\geq \frac{1}{2q_{n_k-i}}.
\]
Therefore by assumption on $x$ we get, 
\begin{equation}\label{log0}
	\begin{split}
\left|\frac{1}{\tilde x_{i+1}}+\sum_{s=1}^{b_{n_k-1-i}-1}\left[(x(i,s))_{n_k-1-i,min}\right]^{-1}\right|&\leq
 \left|\sum_{s=1}^{b_{n_k-1-i}}\frac{1}{\tilde{x}_{i}+\frac{s}{2q_{n_k-i}}}\right|\\
 &\leq 2q_{n_{k}-i}\left|\sum_{s=1}^{b_{n_k-1-i}}\frac{1}{2c''\frac{q_{n_k-i}}{q_{n_k}}+ s} \right|.
 \end{split}
\end{equation}
 Moreover, by the standard estimate on harmonic sum along arithmetic progression, we get 
\begin{equation}\label{log1}
\begin{split}
2q_{n_{k}-i}\left|\sum_{s=1}^{b_{n_k-1-i}}\frac{1}{2c''\frac{q_{n_k-i}}{q_{n_k}}+ s} \right|&\leq 2q_{n_{k}-i}\ln\left(1+\frac{b_{n_k-1-i}q_{n_k}}{2c''q_{n_k-i}} \right)\\
&\le 
 2q_{n_{k}-i}\ln\left((2c''+1)\frac{a_{n_k-1-i}q_{n_k}}{2c''q_{n_k-i}} \right).
\end{split}
\end{equation}
By using the fact that {$q_j=a_jq_{j-1}+q_{j-2}\le a_j(q_{j-1}+q_{j-2})\le 2a_jq_{j-1}\le ea_jq_{j-1}$} and the diophantine condition \eqref{eq:diophcon} we obtain
\begin{equation}\label{log2}
	\begin{split}
\ln\left((2c''+1)\frac{a_{n_k-1-i}q_{n_k}}{2c''q_{n_k-i}} \right)&\le \ln\left( \tfrac{2c''+1}{2c''}e^{i+1}\prod_{j=n_k-1-i}^{n_k}a_j \right)\\
&\le\ln\left(\tfrac{2c''+1}{2c''}e^{i+1}\prod_{j=0}^{i+1}e^{cj} \right)\\
&\le\ln\tfrac{2c''+1}{2c''}+(i+1)+\frac{c(i+2)(i+1)}{2}
	\end{split}
	\end{equation}
By combining \eqref{log0}, \eqref{log1} and \eqref{log2} and the assumption on $x$  we get that there exists a constant $C'''>0$ such that
\[
\begin{split}
\sum_{i=0}^{n_k-1}\sum_{s=0}^{b_{n_k-1-i}}&\left[(x(i,s))_{n_k-1-i,min}\right]^{-1}\\
&\leq\frac{1}{\tilde x_0}+\sum_{i=0}^{n_k-1}2q_{n_k-i}\left(\ln\tfrac{2c''+1}{2c''}+(i+1)+\frac{c(i+2)(i+1)}{2}\right)\\
&\leq C'''q_{n_k}, 
\end{split}
\]
since
$\frac{q_{n_k-s}}{q_{n_k}}\leq 2^{\lfloor-s/2\rfloor}$\footnote{ Indeed notice that for every $n\geq 2$, $q_n=a_nq_{n-1}+q_{n-2}\geq q_{n-1}+q_{n-2}\geq 2q_{n-2}$. Therefore $q_n\geq 2^{\lfloor{s/2}\rfloor}q_{n-s}$. We use it for $n=n_k$.}, for $s\geq 2$.
This together with \eqref{eq:splt} finishes the proof by setting $C=C'(1+C''')$.
\end{proof}

\subsection{Spectral disjointness}
Recall that we consider $\{T^{\alpha,f}_{Kt}\}_{t\in\R}$ and $\{T^{\alpha,f}_{Lt}\}_{t\in\R}$, where $K,L\in \N$, $K\neq L$, $\alpha\in \cC$ and $(n_k)_{k\in \N}$ is the sequence along which \eqref{eq:diophcon} holds for $\alpha$. The distinct natural numbers $K$ and $L$ are fixed in what follows and we will assume WLOG and that $L>K$.

Let $\T\ni x_k:=\frac{1}{2 q_{n_k}}$ and let $c_k:=S_{q_{n_k}}(f)(x_k)$. The following set will play an important role in establishing spectral disjointness: for $b\in \R_+$,  let 
\be\label{eq:cuspset}
A_k(b):=\{x\in \T\;:\; |S_{q_{n_k}}(f)(x)-c_k|\geq b\}.
\ee
Notice that for $b'>b>0$, $A_k(b')\subset A_k(b)$.
The two lemmas below  allow us to show spectral disjointness of $(T_{Kt})$ and $(T_{Lt})$. 

\begin{lemma}\label{lem:tech1} 
There exists  a constant $D>0$ such that for  every $k\in \N$, we have 
$$
D^{-1}e^{-b}\leq \lambda\left(A_k(b)\right)\leq De^{-b}.
$$
\end{lemma}

\begin{lemma}\label{lem:tech2} There exists $k_0\in \N$ and $b_0>0$ such that for every $k\geq k_0$, $b\geq b_0$ and for every $i\in \{-L,-L+1,\ldots, -1,1,\ldots, L\}$, we have
\be\label{eq:lem2}
A_k(b)\cap R_{\alpha}^{i q_{n_k}}\left(A_k\left(\frac{b}{2L}\right)\right)=\emptyset.
\ee
\end{lemma}

We will give proofs of Lemma \ref{lem:tech1} and \ref{lem:tech2} in Subsection \ref{sec:lemproof}. Let us first show how the two lemmas imply the main result of this section. Notice first that Lemma \ref{lem:tech2} has the following important consequence:
\begin{corollary}\label{cor:tech2} Let $k\geq k_0$ and $b\geq b_0$. If $x\in \T$ is such that $|S_{q_{n_k}}(f)(x)-c_k|\geq b$, then for every $i\in \{-L,-L+1,\ldots, -1,1,\ldots, L\}$, $|S_{q_{n_k}}(f)(R_\alpha^{iq_{n_k}}x)-c_k|\leq \frac{b}{2L}$.
\end{corollary}
\begin{proof}This is a straightforward consequence of Lemma \ref{lem:tech2} and \eqref{eq:cuspset} as every $x$ contradicting the statement of corollary has to satisfy 
$$
x\in A_k(b)\cap R_{\alpha}^{i q_{n_k}}\left(A_k\left(\frac{b}{2L}\right)\right).
$$
\end{proof}

\begin{proof}[of Theorem~{\rm\ref{thm:1}}.] To show spectral disjointness of $\{T^{\alpha,f}_{Kt}\}_{t\in\R}$ and $\{T^{\alpha,f}_{Lt}\}_{t\in\R}$ we will use Corollary \ref{cor:maintool} with $a_k=c_k$, $W_k=\T$ and $q_k=q_{n_k}$, $k\in \N$. First, it follows by \cite{FrLem2} that $\{T^{\alpha,f}_t\}_{t\in\R}$ is weakly mixing (the authors show weak mixing for every $\alpha\notin \Q$). Notice also that  with this choice of $\{W_k\}$, $\{a_k\}$ and $\{q_k\}$, \eqref{wl1}-\eqref{wl3} hold trivially. Moreover, by Lemma \ref{lem:tech1} and the definition of $A_k(b)$ it follows that \eqref{wl5} holds. By cocycle identity and \eqref{wl5}, it then follows that for $w\in \{K,L\}$, we have 
$$
\int_\T |S_{wq_{n_k}}(f)(x)-wc_k|^2 d\lambda \;\text{ is bounded }.
$$
Therefore, the sequence of measures $(S_{wq_{n_k}}(f)(x)-wc_k)_\ast\lambda$ on $\R$ is \emph{uniformly tight}. Let then (by passing to subsequence if necessary) 

\be\label{eq:con1}
(S_{Kq_{n_k}}(f)(x)-Kc_k)_\ast\lambda\to P_K
\ee
and 
\be\label{eq:con2}
(S_{Lq_{n_k}}(f)(x)-Lc_k)_\ast\lambda \to P_L.
\ee
For simplicity we will denote $P_K=P$ and $P_L=Q$.

To show spectral disjointness, by Corollary \ref{cor:maintool}, it is enough to show that $P, Q\in \mathcal{P}(\R)$ both  have exponential decay  and that $Res_K(P)\neq Res_L(Q)$. 
To show exponential decay of $P$ and $Q$ we will use the technique from \cite{FrLem2}. In \cite{FrLem2} the measure $P$ was obtained as a weak limit of 
$(S_{q_{n_k}}(f)(x)-c_k)_\ast\lambda$, where $(n_k)_{k\in \N}$ was a {\em very rigid} sequence, i.e. $a_{n_k+1}\geq n_k$. In our case, if we used this very rigid sequence, then one could show that $P=Q$. Therefore we are forced to use {\em balanced sequences}  $(n_k)_{k\in \N}$, (see \eqref{eq:diophcon}), which allows to show that the measures $Res_K(P)$ and $Res_L(Q)$ are different, however it is harder to show exponential decay. 
The proof is henceforward divided into two steps. Firstly we show that $Res_K(P)\neq Res_L(Q)$. Secondly we show the exponential decay of the aforementioned measures.

\noindent\textbf{The measures $Res_K(P)$ and $Res_L(Q)$ are not equal}.
To show that $Res_K(P)\neq Res_L(Q)$ we will show that there exists $b>0$ such that 
$$
Res_K(P)\left((-\infty, -b)\cup(b,+\infty)\right)\neq Res_L(Q)\left((-\infty, -b)\cup(b,+\infty)\right).
$$
By the definition of $P$ and $Q$ and Remark \ref{rem:messet}, the above follows by showing
\begin{multline}\label{eq:shw1}
\lim_{k\to+\infty} \lambda\{x\in \T\;:\; |S_{Kq_{n_k}}(f)(x)-Kc_k|>Kb\}\neq \\
\lim_{k\to+\infty} \lambda\{x\in \T\;:\; |S_{Lq_{n_k}}(f)(x)-Lc_k|>Lb\}.
\end{multline}
Let $w\in \{K,L\}$. Let $b>0$ be a parameter to be specified later, for now we assume $b\geq b_0$, where $b_0$ comes from Lemma \ref{lem:tech2}. Notice that by cocycle identity, we have 
\be\label{eq:coc}
S_{wq_{n_k}}(f)(x)-wc_k=
\sum_{i=0}^{w-1}\left(S_{q_{n_k}}(f)(R_\alpha^{iq_{n_k}}x)-c_k\right).
\ee
Let $B^w_k(b):=\{x\in \T\;:\;  \exists_{i\in\{1,\ldots,w\}}\,  b\leq |S_{q_{n_k}}(f)(R_\alpha^{iq_{n_k}}x)-c_k|\leq (w-1/2)b\}$. We will show below that 
\be\label{eq:empint}
\{x\in \T\;:\; |S_{wq_{n_k}}(f)(x)-wc_k|\geq wb\}\cap B^w_k(b)=\emptyset.
\ee
Indeed, notice that if for some $i\in \{1,\ldots,w\}$, we have  $b\leq |S_{q_{n_k}}(f)(R_\alpha^{iq_{n_k}}x)-c_k|\leq (w-1/2)b$, then, by Corollary \ref{cor:tech2} for $\bar{x}=R_\alpha^{iq_{n_k}}x$, for every $j\in \{1,\ldots,w\}$, $j\neq i$, we have $|S_{q_{n_k}}(f)(R_\alpha^{iq_{n_k}}x)-c_k|\leq \frac{b}{2L}$. This, { the} triangle inequality and \eqref{eq:coc} imply that
$$
 |S_{wq_{n_k}}(f)(x)-wc_k|\leq (w-1/2)b+(w-1)\frac{b}{2L}<wb,
$$
since $w\leq L$. This shows \eqref{eq:empint}. { Notice that if for $x\in \T$ it holds that $|S_{wq_{n_k}}(f)(x)-wc_k|\geq wb$, then by \eqref{eq:coc} and the triangle inequality, there exists $i\in \{1,\ldots,w\}$ such that $|S_{q_{n_k}}(f)(R_\alpha^{iq_{n_k}}x)-c_k|\geq b$.} Therefore and by \eqref{eq:empint}, it follows that (see also \eqref{eq:cuspset})
\begin{multline*}
\{x\in \T\;:\; |S_{wq_{n_k}}(f)(x)-wc_k|\geq wb\}\subset \\
\{x\in \T\;: \;\exists_i\in\{1,\ldots,w\}  |S_{q_{n_k}}(f)(R_\alpha^{iq_{n_k}}x)-c_k|> (w-1/2)b\}\subset\\
\bigcup_{i=1}^wR_\alpha^{-iq_{n_k}}\left(A_k((w-1/2)b\right).
\end{multline*}
On the other hand,
$$
 A_k((w+\frac{w}{2L})b)\subset \{x\in \T\;:\; |S_{wq_{n_k}}(f)(x)-wc_k|\geq wb\}.
$$
Indeed, since $x\in A_k((w+\frac{w}{2L})b)$ iff $|S_{q_{n_k}}(f)(x)-c_k|\ge (w+\frac{w}{2L})b$ which in turn implies that $|S_{q_{n_k}}(f)(x)-c_k|\ge b$, then by \eqref{eq:cuspset} and Corollary \eqref{cor:tech2} we get 
\[
|S_{wq_{n_k}}(f)(x)-wc_k|\geq (w+\frac{w}{2L})b-(w-1)\frac{b}{2L}\geq wb.
\] 
The two above inclusions together with Lemma \ref{lem:tech1} imply that 
$$
D^{-1}e^{-w(1+\frac{1}{2L})b}\leq \lambda\left(\{x\in \T\;:\; |S_{wq_{n_k}}(f)(x)-wc_k|\geq wb\}\right)\leq wDe^{-(w-1/2)b}.
$$
Using this for $w=K$, we get 
$$
\lim_{k\to+\infty} \lambda\{x\in \T\;:\; |S_{Kq_{n_k}}(f)(x)-Kc_k|>Kb\}\geq D^{-1}e^{-K(1+\frac{1}{2L})b}
$$
and for $w=L$, we get 
$$
\lim_{k\to+\infty} \lambda\{x\in \T\;:\; |S_{Lq_{n_k}}(f)(x)-Lc_k|>Lb\}\leq LDe^{-(L-1/2)b}.
$$
Since $L>K$ (and so $L\geq K+1$), we have $D^{-1}e^{-K(1+\frac{1}{2L})b}>LDe^{-(L-1/2)b}$ if $b$ is large enough (depending on $K,L,D$). The two above inequalities prove \eqref{eq:shw1}. Therefore $Res_K(P)\neq Res_L(Q)$.

\textbf{Exponential decay of $P$ and $Q$.}
 Since the proof of exponential decay of $P$ and $Q$ follows the same steps, we will show the proof of exponential decay of $P$(see also Proposition 7 in \cite{FrLem2})\footnote{As mentioned before, the diophantine condition used here to get the limiting measure is orthogonal to the one in \cite{FrLem2}. Therefore, although the statement of Lemma \ref{lem:tech1} is the same as that of Proposition 7 in \cite{FrLem2}, the proofs are different.}, that is
 there exist constants $\bar{C}, \bar{c}>0$ such that 
$$
P\left(\{t\in \R \;:\; |t|\geq s\}\right)\leq \bar{C}e^{-s\bar{c}}.
$$
Indeed, notice that by \eqref{eq:con1}, the cocycle indentity (see \eqref{eq:coc}) and \eqref{eq:cuspset}, we have 

\begin{equation*}
		\begin{split}
P\left(\{t\in \R \;:\; |t|\geq s\}\right)&\leq P\left(\{t\in \R \;:\; |t|> \tfrac{1}{2}s\}\right)\\ &\leq\liminf_{k\to +\infty} \lambda\left(\{x\in \T\;:\; |S_{Kq_{n_k}}(f)(x)-Kc_k|>\tfrac{1}{2}s\}\right)\\ &\leq 
\limsup_{k\to +\infty} \lambda\left(\{x\in \T\;:\;\exists_{i\in \{1,\ldots K\}}\; |S_{q_{n_k}}(f)(R_\alpha^{iq_{n_k}}x)-c_k|\geq \frac{s}{2K}\}\right)\\ &\leq 
\limsup_{k\to+\infty}\lambda\left(\bigcup_{i=0}^{K-1}R_\alpha^{-iq_{n_k}}(A_k(\tfrac{s}{2K}))\right)\leq 2KDe^{-\frac{s}{2K}},
\end{split}
\end{equation*}
the last inequality by Lemma \ref{lem:tech1}. This finishes the proof.
\end{proof}

\subsection{Proof of Lemma \ref{lem:tech1} and \ref{lem:tech2}}\label{sec:lemproof}
Lemma \ref{lem:tech1} and \ref{lem:tech2} are consequences of the following lemma:
\begin{lemma}\label{lem:new}There exists $D,b_0,k_0>0$ such that for every $b\geq b_0$ and $k\geq k_0$, we have
\begin{multline}\label{eq:oldin}
\bigcup_{j=0}^{q_{n_k}-1}\left[-\frac{D^{-1}e^{-b}}{2q_{n_k}}-j\alpha,-j\alpha+\frac{D^{-1}e^{-b}}{2q_{n_k}}\right]\subset A_k(b)
\subset\\
\bigcup_{j=0}^{q_{n_k}-1}\left[-\frac{De^{-b}}{2q_{n_k}}-j\alpha,-j\alpha+\frac{De^{-b}}{2q_{n_k}}\right].
\end{multline}
\end{lemma}
Before we prove Lemma \ref{lem:new} let us show how it implies Lemma \ref{lem:tech1} and \ref{lem:tech2}.
\begin{proof}[of Lemma~{\rm\ref{lem:tech1}}]
The upper bound in Lemma \ref{lem:tech1} follows strightforward from the second inclusion in Lemma \ref{lem:new}. For the lower bound, by the first inclusion in Lemma \ref{lem:new}, it is enough to show that for $v,u\in \{0,\ldots,q_{n_k}-1\}$, $v\neq u$, we have
$$
\left[-\frac{D^{-1}e^{-b}}{2q_{n_k}}-v\alpha,-v\alpha+\frac{D^{-1}e^{-b}}{2q_{n_k}}\right]\cap \left[-\frac{D^{-1}e^{-b}}{2q_{n_k}}-u\alpha,-u\alpha+\frac{D^{-1}e^{-b}}{2q_{n_k}}\right]=\emptyset.
$$
This however follows from the fact that $D^{-1}e^{-b}<1/2$ and 
$$
\|v\alpha-u\alpha\|=\|(v-u)\alpha\|\geq \min_{0< j\leq q_{n_k}-1}\|j\alpha\|\geq \frac{1}{2q_{n_k}}.
$$
This finishes the proof.
\end{proof}

Now let us prove Lemma \ref{lem:tech2}:

\begin{proof}[of Lemma~{\ref{lem:tech2}}]
Let $b>2Lb_0$ ($b_0$ from Lemma \ref{lem:new}) and let $x\in A_k(\frac{b}{2L})$. By the right inclusion in Lemma \ref{lem:new} it follows that there exists $j\in\{0,\ldots, q_{n_k}-1\}$ such that 
\be\label{eq:as1}
\|x-(-j\alpha)\|\leq \frac{De^{\frac{-b}{2L}}}{2q_{n_k}}.
\ee
We will show that for every $i\in \{-L,\ldots,-1,1,\ldots, L\}$, we have
\be\label{eq:ns}
\min_{0\leq s\leq q_{n_k-1}}\|R_\alpha^{iq_{n_k}}x-(-s\alpha)\|\geq \frac{2De^{-b}}{2q_{n_k}}.
\ee
This together with the right inclusion in Lemma \ref{lem:new} implies that for every \\
$i\in \{-L,\ldots,-1,1,\ldots, L\}$ we have $R_\alpha^{iq_{n_k}}x\notin A_k(b)$ and therefore the proof of Lemma \ref{lem:tech2} is completed. So we only need to show \eqref{eq:ns}. Fix $i\in \{-L,\ldots,-1,1,\ldots, L\}$. 
Notice that by \eqref{eq:diophcon}, we have $\frac{q_{n_k+1}}{q_{n_k}}\leq a_{n_k+1}+1\leq 300(K^2+L^2)$. Thus by \eqref{eq:as1} we obtain 
\begin{multline}\label{eq:nee1}
\|R_\alpha^{iq_{n_k}}x-(-j\alpha)\|=\|x-(-j\alpha)+iq_{n_k}\alpha\|\geq 
|i|\|q_{n_k}\alpha\|-\|x-(-j\alpha)\|\geq\\
 \frac{1}{2q_{n_k+1}}-\frac{De^{\frac{-b}{2L}}}{2q_{n_k}}\geq \frac{{(300(K^2+L^2))^{-1}}-De^{\frac{-b}{2L}}}{2q_{n_k}}\geq \frac{2De^{-b}}{2q_{n_k}},
\end{multline}
by enlarging $b$ if necessary. This shows \eqref{eq:ns} for $s=j$. Assume now $0\leq s\leq q_{n_k}-1$ and $s\neq j$. By \eqref{eq:diophcon}, we have $\frac{q_{n_k+1}}{q_{n_k}}\geq a_{n_k+1}\geq 100(K^2+L^2)$. Moreover $\|(s-j)\alpha\|\geq \min_{0\leq r\leq q_{n_k}-1}\|r\alpha\|\geq \frac{1}{q_{n_k}}$. Therefore and by \eqref{eq:as1}, we have
\begin{multline}\label{eq:nee2}
\|R_\alpha^{iq_{n_k}}x-(-s\alpha)\|=\|(s-j)\alpha +x-(-j\alpha)+iq_{n_k}\alpha\|\geq \\
\|(s-j)\alpha\|- |i|\|q_{n_k}\alpha\|-\|x-(-j\alpha)\|\geq 
\frac{1}{2q_{n_k}}-\frac{|i|}{100(K^2+L^2)q_{n_k}}-\frac{De^{\frac{-b}{2L}}}{2q_{n_k}}\geq \frac{2De^{-b}}{2q_{n_k}},
\end{multline}
the last inequality since $|i|\leq L$ and (by enlarging $b$ if necessary), 
$1-\frac{2|i|}{200(K^2+L^2)}-De^{\frac{-b}{2L}}>2De^{-b}$.
This finishes the proof of \eqref{eq:ns} for $s\neq j$. By \eqref{eq:nee1} and \eqref{eq:nee2} it follows that \eqref{eq:ns} holds. This finishes the proof of Lemma \ref{lem:tech2}.
\end{proof}

So it only remains to prove Lemma \ref{lem:new}:

\begin{proof}[of Lemma~{\rm\ref{lem:new}}]
Let $h:\T\to \R_+$, $h(x)=-(\log x +\log(1-x))$. Then (see \eqref{eq:logsym} and recall that we assume that $C_f=1$) $f(x)=h(x)+g(x)$, where $g\in C^3(\T)$. By Denjoy-Koksma inequality (and triangle inequality), it follows that for every $x,y\in \T$, 
\be\label{eq:dkd}
|S_{q_{n_k}}(g)(x)-S_{q_{n_k}}(g)(y)|<2{\rm Var} (g).
\ee
Recall also that $c_k:=S_{q_{n_k}}(h)(x_k)+S_{q_{n_k}}(g)(x_k)$. Let $c_k^h:=S_{q_{n_k}}(h)(x_k)$.

Let 
\be\label{eq:defV}V_k^h(b):=\{x\in \T\;:\; |S_{q_{n_k}}(h)(x)-c_k^h|\geq b\}.
\ee
 We claim that for $b\geq 4{\rm Var}(g)$, we have 
\be\label{eq:neweq}
V_k^h\left(b+4{\rm Var}(g)\right)\subset A_k(b)\subset V_k^h\left(b-4{\rm Var}(g)\right).
\ee
Indeed, this just follows from $S_{q_{n_k}}(f)(x)-c_k=[S_{q_{n_k}}(h)(x)-c_k^h]+[S_{q_{n_k}}(g)(x)-S_{q_{n_k}}(g)(x_k)]$ and from \eqref{eq:dkd} (for $x$ and $x_k$). By \eqref{eq:neweq}, Lemma \ref{lem:new} follows by showing that there exists $D',b_0',k_0'>0$  such that for $k\geq k'_0$ and $b'\geq b'_0$, we have 
\begin{multline}\label{eq:newin}
\bigcup_{j=0}^{q_{n_k}-1}\left[-\frac{D'^{-1}e^{-b'}}{2q_{n_k}}-j\alpha,-j\alpha+\frac{D'^{-1}e^{-b'}}{2q_{n_k}}\right]\subset V^h_k(b')\subset\\
\bigcup_{j=0}^{q_{n_k}-1}\left[-\frac{D'e^{-b'}}{2q_{n_k}}-j\alpha,-j\alpha+\frac{D'e^{-b'}}{2q_{n_k}}\right].
\end{multline}
Indeed, by the left inclusion in \eqref{eq:neweq} for $b'=b+4{\rm Var}(g)$ and then the right inclusion in \eqref{eq:neweq} for $b'=b-4{\rm Var}(g)$) it then follows that \eqref{eq:oldin} holds with $D=D'e^{4{\rm Var}(g)}$, $b_0=b'_0+4{\rm Var}(g)$ and $k_0=k_0'$.
Therefore it is enough to show \eqref{eq:newin}. To simplify the notation we will drop all the apostrophes in the proof of \eqref{eq:newin}.
\bigskip

Consider the partition  $\{I_1,\ldots, I_{q_{n_k}}\}$ of $\T$ by points $\{-i\alpha\}_{i=0}^{q_{n_k}-1}$ and denote $I_s=(v_s,w_s)$. Then \eqref{eq:newin} is a straightforward consequence of the following: there exists $D,k_0,b_0$ such that for every $k\geq k_0$, $b\geq b_0$ and every $s\in \{1,\ldots,q_{n_k}\}$, we have 
\begin{multline}\label{eq:newin2}
  \left(v_s,v_s+ \frac{D^{-1}e^{-b}}{2q_{n_k}}\right)\cup \left(w_s-\frac{D^{-1}e^{-b}}{2q_{n_k}},w_s\right) \subset V_k^h(b)\cap I_s\subset\\
   \left(v_s,v_s+ \frac{De^{-b}}{2q_{n_k}}\right)\cup \left(w_s-\frac{De^{-b}}{2q_{n_k}},w_s\right).
\end{multline}
Therefore it is enough to show \eqref{eq:newin2}. Let $y_s\in I_s$ be the midpoint of $I_s$, i.e. $y_s:=v_s+\frac{w_s-v_s}{2}$ and let 
$$
W_{k,s}(b):=\{x\in I_s\;:\;|S_{q_{n_k}}(h)(x)-S_{q_{n_k}}(h)(y_s)|\geq b\}.
$$
Notice that by \eqref{eq:diophcon}, we have $q_{n_k+1}\geq 100(K^2+L^2)q_{n_k}\geq 200 q_{n_k}$.
Below we will also use the following estimates on the length of $I_s$: for every $s\in\{1,\ldots, q_{n_k}-1\}$, we have  (see \eqref{eq:lenn}) 
\be\label{eq:lenin}
\frac{99}{100q_{n_k}}\leq \frac{1}{q_{n_k}}-\frac{2}{q_{n_k+1}}\leq |I_s|\leq \frac{1}{q_{n_k}}+\frac{2}{q_{n_k+1}}\leq \frac{101}{100q_{n_k}}.
\ee

We have the following two claims:\\
\textbf{Claim A.} There exists $D,k_0,b_0$ such that for every $k\geq k_0$, $b\geq b_0$ and every $s\in \{1,\ldots,q_{n_k}\}$, we have 
\begin{multline}\label{eq:newin3}
  \left(v_s,v_s+ \frac{D^{-1}e^{-b}}{2q_{n_k}}\right)\cup \left(w_s-\frac{D^{-1}e^{-b}}{2q_{n_k}},w_s\right) \subset W_{k,s}(b)\subset \\
   \left(v_s,v_s+ \frac{De^{-b}}{2q_{n_k}}\right)\cup \left(w_s-\frac{De^{-b}}{2q_{n_k}},w_s\right).
\end{multline}
\textbf{Claim B.} There exists a (global) constant $C'''>0$ such that for every $k\in \N$, every  $s\in\{1,\ldots,q_{n_k}\}$ and every $b>0$, we have 
\be\label{eq:wv}
W_{k,s}(b+C''')\subset V_k^h(b)\cap I_s\subset W_{k,s}(b-C''').
\ee
Notice that \textbf{Claim A.} and \textbf{Claim B.} together imply \eqref{eq:newin2}: the left inclusion in \eqref{eq:newin3} for $b-C'''$ together with left inclusion in \eqref{eq:wv} imply that 
$$
\left(v_s,v_s+ \frac{D^{-1}e^{-C'''}e^{-b}}{2q_{n_k}}\right)\cup \left(w_s-\frac{D^{-1}e^{-C'''}e^{-b}}{2q_{n_k}},w_s\right)\subset V_k^h(b)\cap I_s.
$$
Analogously, the right inclusion in \eqref{eq:newin3}  for $b+C'''$ together with right inclusion in \eqref{eq:wv} imply that 
$$
V_k^h(b)\cap I_s\subset\\
   \left(v_s,v_s+ \frac{De^{C'''}e^{-b}}{2q_{n_k}}\right)\cup \left(w_s-\frac{De^{C'''}e^{-b}}{2q_{n_k}},w_s\right).
$$
This finishes the proof of \eqref{eq:newin2} and hence also the proof of Lemma \ref{lem:new}. So it remains to prove \textbf{Claim A.} and \textbf{Claim B.}

\noindent\emph{Proof of \textbf{Claim A}.} 
By analyticity of $\log(\cdot)$ and $\log(1-\cdot)$ on $(0,1)$, we have for $s=1,\ldots,q_{n_k}$ and $x\in I_s$
\be\label{eq:najwazniejszywzor}
S_{q_{n_k}}(h)(x)-S_{q_{n_k}}(h)(y_{s})=\sum_{r=1}^{+\infty}\frac{S_{q_{n_k}}(h^{(r)})(y_{s})}{r!}(x-y_{s})^r
\ee
Since $y_s$ is the midpoint of $I_s=(v_s,w_s)$ and the endpoints of $I_s$ are in the in the orbit of $\alpha$ it follows that $y^+_{n_k,min}=y^-_{n_k,min}=\frac{|I_s|}{2}$.
Therefore by Lemma \ref{lem:DKlog} for $r\geq 1$, we have 
\be\label{eq:neab}
\left|S_{q_{n_k}}(h^{(r)})(y_{s})-[(-1)^{r}+1](r-1)!\left(\frac{2}{|I_s|}\right)^r\right|
|\leq 100(r-1)!\|q_{n_k-1}\alpha\|^{-r}.
\ee
Notice that by \eqref{eq:diophcon} it follows that $\|q_{n_k-1}\alpha\|\geq \frac{98}{100q_{n_k}}$ (indeed, $\|q_{n_k}\alpha\|{ q_{n_k-1}}\le \|q_{n_k}\alpha\|\frac{q_{n_k+1}}{a_{n_{k}+1}}\le\frac{1}{100}$). This together with \eqref{eq:lenin} and again \eqref{eq:diophcon} implies that $\frac{|x-y_s|}{\|q_{n_k-1}\alpha\|}\leq \frac{100q_{n_k}|I_s|}{2\cdot 98}\leq \frac{101}{196}<1$. Therefore and by \eqref{eq:najwazniejszywzor} and \eqref{eq:neab}, there exists $C''''>0$ such that
$$
\left|S_{q_{n_k}}(h)(x)-S_{q_{n_k}}(h)(y_{s})+\sum_{r=1}^{+\infty}\frac{(-1)^{r+1}-1}{r}\left(\frac{2(x-y_{s})}{|I_s|}\right)^r\right|\leq C''''.
$$
This is equivalent to 
$$
\left|[S_{q_{n_k}}(h)(x)-S_{q_{n_k}}(h)(y_{s})]+
\log\left(1+\frac{2(x-y_{s})}{|I_s|}\right)+\log\left(1-\frac{2(x-y_{s})}{|I_s|}\right)\right|\leq C'''',
$$
which, since $I_s=(v_s,w_s)$ and $y_s$ is the midpoint of $I_s$ is equivalent to 
\be\label{eq:host}
\left|[S_{q_{n_k}}(h)(x)-S_{q_{n_k}}(h)(y_{s})]+
\log\left(\frac{2(x-v_s)}{|I_s|}\right)+\log\left(\frac{2(w_s-x)}{|I_s|}\right)\right|\leq C''''.
\ee
Let $W^{\log}_s(b):=\{x\in I_s\;:\; -\log\left(\frac{2(x-v_s)}{|I_s|}\right)-\log\left(\frac{2(w_s-x)}{|I_s|}\right)\geq b\}$. Then by \eqref{eq:host} and the definition of $W_{k,s}(b)$, we have
\be\label{eq:host2}
W^{\log}_s(b+C'''')\subset W_{k,s}(b)\subset W^{\log}_s(b-C'''').
\ee
Moreover, $x\in W^{\log}_s(b+C'''')$ iff $\log\left(\frac{2(x-v_s)}{|I_s|}\right)+\log\left(\frac{2(w_s-x)}{|I_s|}\right)\leq -b-C''''$, which holds if  $|x-v_s|\leq D^{-1}e^{-b}|I_s| $ or if $|w_s-x|\leq D^{-1}e^{-b}|I_s|$ for some (global) constant $D>1$ ($D$ depending on $C''''$). By \eqref{eq:lenin}, it follows that (by enlarging $D$ if necessary)
\be\label{eq:host3}
\left(v_s,v_s+\frac{D^{-1}e^{-b}}{2q_{n_k}}\right)\cup\left(w_s-\frac{D^{-1}e^{-b}}{2q_{n_k}},w_s\right)\subset W^{\log}_s(b+C'''').
\ee
Similarly (by enlarging $D$ if necessary), it follows that for $x\in I_s$ if $|x-v_s|\geq De^{-b}|I_s|$ and $|w_s-x|\geq De^{-b}|I_s|$, then $x\notin W^{\log}_s(b-C'''')$. Therefore 
\be\label{eq:host4}
W^{\log}_s(b-C'''')\subset \left(v_s,v_s+\frac{De^{-b}}{2q_{n_k}}\right)\cup\left(w_s-\frac{De^{-b}}{2q_{n_k}},w_s\right).
\ee

Finally, \eqref{eq:host2}, \eqref{eq:host3} and \eqref{eq:host4} give \eqref{eq:newin3} and hence finish the proof of \textbf{Claim A}.

\noindent\emph{Proof of \textbf{Claim B}.} 
We will first show that there exists a (global) constant $C'''>0$ such that for every $k\in \N$ and every $s\in \{1,\ldots, q_{n_k}\}$, we have
\be\label{eq:rts}
|S_{q_{n_k}}(h)(y_s)-c_k^h|<C'''.
\ee
Recall that  $x_k=\frac{1}{2q_{n_k}}$ and $c_k^h:=S_{q_{n_k}}(h)(x_k)$. The left endpoint of $I_s=(v_s,w_s)$ is equal to $-\ell\alpha$ for some $\ell\in [0,\ldots,q_{n_k}-1]$. Consider the point $x_k-\ell \alpha$. By \eqref{eq:lenin}, it follows that $x_k-\ell\alpha\in I_s$, and moreover, again by  \eqref{eq:lenin}
it follows that 
\begin{equation}\label{eq:1/4}
\min\{|x_k-\ell\alpha-v_s|,|w_s-(x_k-\ell\alpha)|\}\geq\frac{49}{100q_{n_k}} \geq\frac{1}{4q_{n_k}}. 
\end{equation}
By cocycle identity and mean value theorem, we have 
\begin{equation}\label{eq:funy}
	\begin{split}
	|c_k^h-S_{q_{n_k}}(h)(x_k-\ell\alpha)|&=|S_{\ell}(h)(x_k-\ell\alpha)-S_{\ell}(h)(x_k-\ell\alpha+q_{n_k}\alpha)|\\&=|S_{\ell}(h')(\theta)|\|q_{n_k}\alpha\|,
\end{split}
\end{equation}
where $\theta\in [\min(x_k-\ell\alpha,x_k-\ell\alpha+q_{n_k}\alpha),\max(x_k-\ell\alpha,x_k-\ell\alpha+q_{n_k}\alpha)]$. Notice that since $\|q_{n_k}\alpha\|\leq \frac{1}{q_{n_k+1}}\leq \frac{1}{100q_{n_k}}$ (see \eqref{eq:diophcon}), by \eqref{eq:1/4} it follows that $\theta_{n,min}\geq \frac{1}{8q_{n_k}}$. Therefore by Lemma \ref{lem:Best} for $c''=1/8$, it follows that $|S_{\ell}(h')(\theta)|\leq C_2q_{n_k}$. Hence, by \eqref{eq:funy}, we have 
\be\label{eq:nesr}
|c_k^h-S_{q_{n_k}}(h)(x_k-\ell\alpha)|\le C_2.
\ee
Moreover, for some $\theta'\in [\min(y_s,x-\ell\alpha),\max(y_s,x-\ell\alpha)]$, we have 
\be\label{eq:nesr2}
|S_{q_{n_k}}(h)(y_s)-S_{q_{n_k}}(h)(x_k-\ell\alpha)|=|S_{q_{n_k}}(h')(\theta')|\cdot|I_s|.
\ee
Similarly to the estimates on $\theta$, it follows that $\theta'_{n_k,min}\geq \frac{1}{8q_{n_k}}$ ($y_s$ is the midpoint of $I_s$ and $x-\ell\alpha$ is $\frac{1}{2q_{n_k}}$ distant from the left endpoint of $I_s$). Therefore, by Lemma \ref{lem:DK} for $\theta'$, \eqref{eq:nesr2} and \eqref{eq:lenin} imply that 
$$
|S_{q_{n_k}}(h)(y_s)-S_{q_{n_k}}(h)(x_k-\ell\alpha)|\leq C_3.
$$
This finishes the proof of \eqref{eq:rts}. By \eqref{eq:rts}  and triangle inequality it follows that if for $x\in I_s$ we have
$|S_{q_{n_k}}(h)(x)-S_{q_{n_k}}(h)(y_s)|\geq b+C'''$, then $|S_{q_{n_k}}(h)(x)-c_k^h|\geq b$ and similarly if $x\in I_s$ satisfies $|S_{q_{n_k}}(h)(x)-c_k^h|\geq b$, then $|S_{q_{n_k}}(h)(x)-S_{q_{n_k}}(h)(y_s)|\geq b-C'''$. This, by the definition of $W_{k,s}(b)$ and $V_k^h(b)$ implies \eqref{eq:wv} and finishes the proof of \textbf{Claim B.}

The proof of Lemma \ref{lem:new} is thus finished.

\end{proof}

\section{Flows over IETs under piecewise constant roof function}\label{sec:pico}
In this section we will prove Theorem \ref{Main}. The proof of this result is long and divided into sections. It is worth to mention that the construction of sequence of towers which is one of the main points of the proof is very similar to the construction given in \cite{BFr} where authors used it to differentiate between special flows over IETs and their inverses. The construction used there is however insufficient in our case because the authors in \cite{BFr} did not control the form of the whole limit of the just a large part of it. The measure of the part they did not control however, in our case would be the same as the part which differentiates between limit measures for different time scalings of considered flows. The construction we present in this proof is more complicated, it allows though to compute the whole limit measure, which in the past was achievable only for rotations. 	

\begin{proof}[of Theorem~{\rm\ref{Main}}.]
	Fix a Rauzy graph $\mathfrak{R}$ of irreducible permutations. Then for almost every $(\pi,\la)\in\mathfrak{R}\times\Lambda^{\mathcal A}$ there exists a natural number $N(\pi,\la)$ such that the $N(\pi,\la)$-th Rauzy matrix $A^{N(\pi,\la)}(\pi,\la)$ is a positive matrix. Fix such $(\pi_0,\la_0)$ and denote $N:=N(\pi_0,\la_0)$ and $B:=A^{N}(\pi_0,\la_0)$. Moreover, assume that $\hat\pi=\pi^{N}_0$ satisfies
	\[
	\hat\pi_0^{-1}(1)=\hat\pi_1^{-1}(d)\ \text{and}\ \hat\pi_0^{-1}(d)=\hat\pi_1^{-1}(1).
	\]
	This is possible, because in view of Theorem \ref{twpierost} every Rauzy graph contains such permutation.
	
	For any positive $\,d\times d$ matrix $B$ let
	\[
	\rho(B)=\max_{1\le i,j,k\le d}\frac{B_{ij}}{B_{ik}}\ge 1.
	\]
	Set $b_{j}=\sum_{i=1}^dB_{ij}$ and let $A$ be any nonnegative
	nonsingular $\,d\times d$ matrix. The following properties are well-known and 
	easy to prove
	\begin{equation}\label{ro1}
	b_j\le\rho(B)b_k\quad\text{ for any }\quad1\le j,k\le d,
	\end{equation}
	\begin{equation}\label{ro2}
	\rho(AB)\le\rho(B).
	\end{equation}
	
	\subsection{The choice of a set of recurrence.}\label{sec:rec} Let $K,L\in\N\setminus\{0\}$ be two distinct numbers such that $K>L$. Let 
	\begin{equation}\label{defep}
	0<\ep<\min\big(\frac{1}{100\rho(B)},\frac{1}{100K}\big)
	\end{equation}
	and 
	\begin{equation}\label{defde}
	\frac{\ep}{3}<\de'<\de<\frac{\ep}{2}.
	\end{equation}
	
	For every $n\in\N$ and $\la\in\Lambda^{\mathcal A}$ define 
	\[
	Q_n(\la)=\la_{\hat\pi_0^{-1}(1)}-(n-2)\big(\sum_{j=2}^d\la_{\hat\pi_0^{-1}(j)} \big).
	\]
	Let $Y^{K}_n\subset\Lambda^{\mathcal A}$ be the set of vectors $\la$ satisfying $Q_n(\la)>0$ and the following conditions 
	\begin{equation}\label{ineq1}
	\frac{1}{2}-\de+\frac{3(\de-\de')}{8}<\frac{\la_{\hat\pi_0^{-1}(1)}-(n-1)\big(\sum_{j=2}^d\la_{\hat\pi_0^{-1}(j)}\big)}{Q_n(\la)}<\frac{1}{2}-\de+\frac{5(\de-\de')}{8},
	\end{equation} 
	\begin{equation}\label{ineq2}
	\frac{1}{2}+\de'+\frac{3(\de-\de')}{8}<\frac{\la_{\hat\pi_0^{-1}(d)}}{Q_n(\la)}<\frac{1}{2}+\de'+\frac{5(\de-\de')}{8},
	\end{equation}
	\begin{equation}\label{ineq3}
	\frac{1}{Q_n(\la)}\sum_{j=2}^{d-1}\la_{\hat\pi_0^{-1}(j)}<\frac{1}{n+\rho(B)}.
	\end{equation}
	
	It is an open set and non-empty. Indeed, if the vector $\hat\la\in\R^{\mathcal A}$ is given by
	\[
	\hat\la_{\hat\pi_0^{-1}(2)}=\ldots=\hat\la_{\hat\pi_0^{-1}(d-1)}:=\frac{\de-\de'}{(d-1)(n+\rho(B))};
	\]
	\[
	\hat\la_{\hat\pi_0^{-1}(d)}:=\frac{1}{2}+\de'+\frac{3}{8}(\de-\de')+\frac{\de-\de'}{(d-1)(n+\rho(B))};
	\]
	and
	\[
	\hat\la_{\hat\pi_0^{-1}(1)}:=\frac{1}{2}-\de+\frac{5}{8}(\de-\de')-2\frac{\de-\de'}{(d-1)(n+\rho(B))}
	+(n-1)\sum_{j=2}^d\hat\la_{\hat\pi_0^{-1}(j)},
	\]
	then $\hat\la/|\hat\la|\in Y^{K}_n$.
	
	We will denote 
	\be\label{eq:tilte}
	\tilde\la_{\hat\pi_0^{-1}(1)}:=\la_{\hat\pi_0^{-1}(1)}-(n-1)\big(\sum_{j=2}^d\la_{\hat\pi_0^{-1}(j)}\big).
	\ee
	Then in view of \eqref{ineq1} and \eqref{ineq2} we obtain
	\begin{equation}\label{ineq4}
	\frac{\ep}{2}<\frac{2\ep}{3}-\frac{\ep}{16}<\de'+\de-\frac{\de-\de'}{4}
	<\frac{\la_{\hat\pi_0^{-1}(d)}-\tilde\la_{\hat\pi_0^{-1}(1)}}{Q_n(\la)}<\de'+\de+\frac{\de-\de'}{4}<2\de<\ep.
	\end{equation}
	
	Define
	\[
	V_n^{K}:=\{(\pi_0,\frac{1}{|B\la|}B\la);\,\la\in Y_n^{K}\}.
	\]
	Since $Y_n^{K}$ is a non-empty open set, then so is $V_n^{K}$.
	Since the normalised Rauzy-Veech induction $\tilde R$ is ergodic and recurrent (see Theorem \ref{powr}), there exists a set of full Lebesgue measure $\Upsilon\subset\mathfrak{R}\times\Lambda^{\mathcal A}$ such that for every $(\pi,\la)\in\Upsilon$ and every $K,L\in\N\setminus\{0\}$ there exists an increasing sequence of natural numbers $\{r_n^{K}\}_{n\in\N}$ satisfying
	\[
	\tilde R^{r_n^{K}-N-(d-1)(n-1)}(\pi,\la)\in \{\hat\pi\}\times V^{K}_n.
	\]
	Then in view of Remark \ref{mnozmac} we also have
	\be\label{defrn}
	\tilde{R}^{r_n^{K}-(d-1)(n-1)}(\pi,\la)\in \{\hat\pi\}\times Y_n^{K}.
	\ee
	
	\begin{figure}[h]
		\includegraphics[scale=0.15]{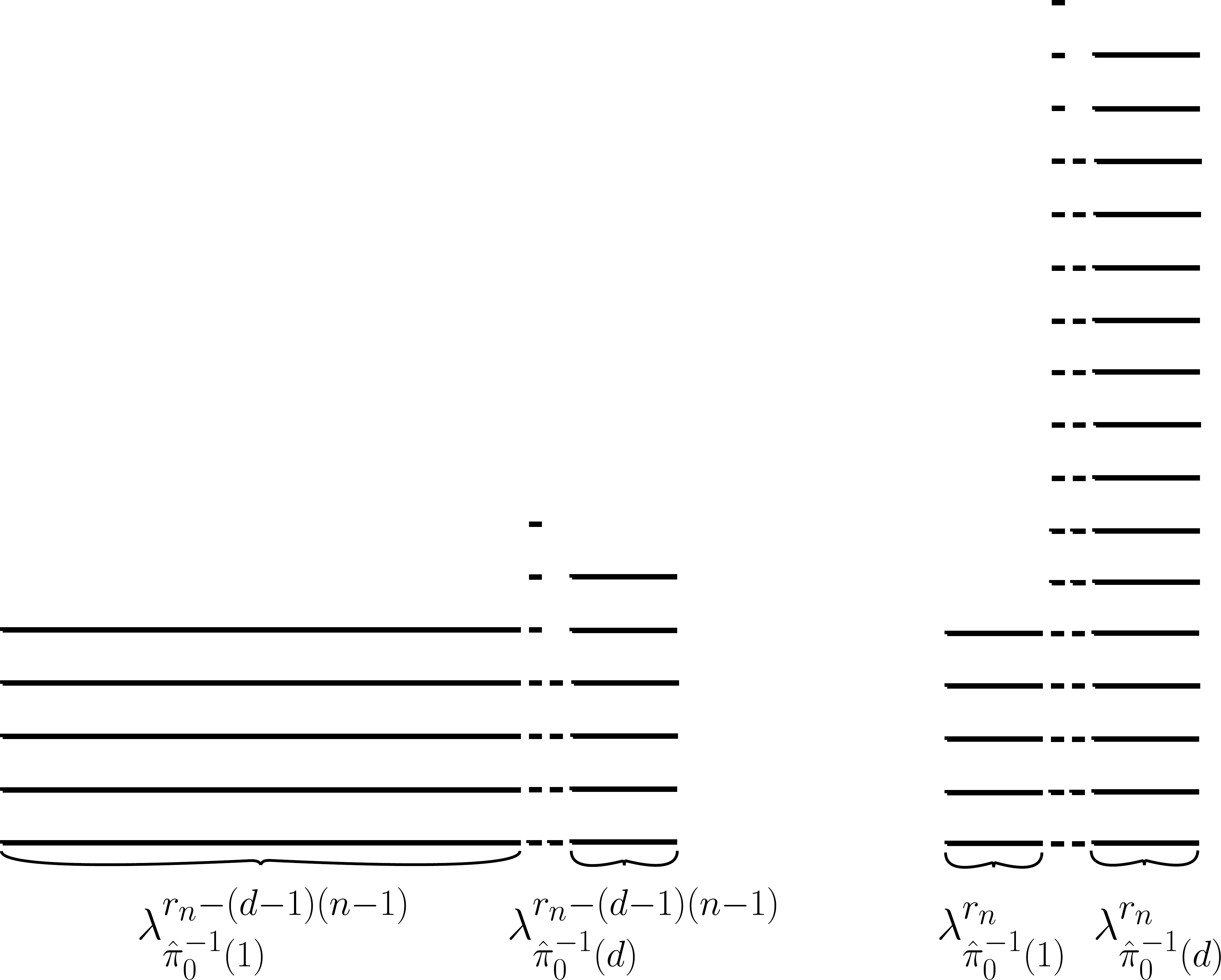}
		$\quad\quad\quad\quad$
		\caption{The tower after $r_n-(d-1)(n-1)$ steps of Rauzy-Veech induction (left) and after $r_n$ steps of Rauzy-Veech induction (right).\label{rys1}}
	\end{figure}
	
	\subsection{Decomposition into towers.}
	{ Fix $(\pi,\la)\in\Upsilon$ and $K,L\in\N$, $K>L$. Moreover, denote by $T$ the IET corresponding to $(\pi,\la)$ and write $r_n$ instead of $r_n^{K}$ (which is the sequence defined in \eqref{defrn}).  For any $n\in\N$ consider ${R}^{r_n}(\pi,\la)={R}^{(d-1)(n-1)}\circ{R}^{r_n-(d-1)(n-1)}(\pi,\la)$. 
	From now on we write $I^n$ instead of $I^{r_n}$ for the domain of ${R}^{r_n}(\pi,\la)$ and $I^n_a$ instead of $I^{r_n}_a$ for the exchanged interval corresponding to $a\in\mathcal A$. Let $s_a^n:=\sum_{b\in\mathcal A}A^{r_n-(d-1)(n-1)}_{ba}(\pi,\la)$ for any $a\in\mathcal A$. 
	Recall that by \eqref{ro1} and \eqref{ro2} since $A^{r_n-(d-1)(n-1)}(\pi,\la)=CB$, where $C$ is a non-negative matrix, {we have \be\label{eq:distest}
		s_a^n\le \rho(B)s_b^n\;\text{ for all }a,b\in\mathcal A. 
		\ee}
	We now show that we obtain the following decomposition of $I=[0,1)$ into Rokhlin towers.
	\begin{itemize}
		\item The domain $I^n$ of ${R}^{r_n}(\pi,\la)$ has length $Q_n(\la^{r_n-(d-1)(n-1)})$;\\
		\item The permutation of IET corresponding to ${R}^{r_n}(\pi,\la)$ is $\hat\pi$;\\
		\item The interval $I^n_{\hat\pi_0^{-1}(1)}$ has length $\la_{\hat\pi_0^{-1}(1)}^{r_n}=\la_{\hat\pi_0^{-1}(1)}^{r_n-(d-1)(n-1)}-(n-1)\big(\sum_{j=2}^d\la_{\hat\pi_0^{-1}(j)}^{r_n-(d-1)(n-1)} \big)$ and the height of the tower obtained by Rauzy-Veech induction is $s_1^n:=s_{\hat\pi_0^{-1}(1)}^n$;\\
		\item Intervals $I^{n}_{\hat\pi_0^{-1}(j)}$ for $j=2,\ldots,d-1$ have lengths $\la_{\hat\pi_0^{-1}(j)}^{r_n}:=\la_{\hat\pi_0^{-1}(j)}^{r_n-(d-1)(n-1)}$ respectively and if $s_j^n:=s_{\hat\pi_0^{-1}(j)}^n$ then the corresponding towers have height equal to $s_j^n+(n-1)s_1^n$;\\
		\item The interval $I^n_{\hat\pi_0^{-1}(d)}$ has length $\la_{\hat\pi_0^{-1}(d)}^{r_n}=\la_{\hat\pi_0^{-1}(d)}^{r_n-(d-1)(n-1)}$ and if $s_d^n:=s_{\hat\pi_0^{-1}(d)}^n$ then the height of the corresponding tower is $s_d^n+(n-1)s_1^n$.
	\end{itemize}
	We show this for $n=2$, for $n\ge 3$ the proof is analogous by proceeding inductively. In view of Remark \ref{rozkladnawieze} we have that $R^{r_2-(d-1)}(\pi,\la)$ is an IET $(\hat\pi,\la^{r_2-(d-1)})$, where $\la^{r_2-(d-1)}$ satisfies conditions \eqref{ineq1}, \eqref{ineq2} and \eqref{ineq3} and the interval $I$ is decomposed into towers over intervals $I_{\hat\pi_0^{-1}(j)}^{r_2-(d-1)}$ of height $s_j^2$ for every $j=1,\ldots,d$ (see left picture in Figure \ref{rys1}). Note that in particular by \eqref{ineq1} we obtain that
	\[
	\la_{\hat\pi_0^{-1}(1)}^{r_2-(d-1)}>\sum_{j=2}^d\la_{\hat\pi_0^{-1}(j)}.
	\]
	Thus for $k=1,\ldots,d-1$, $k$-th step of Rauzy-Veech induction is the first return maps to the interval $\Big[0,|\la^{r_2-(d-1)}|-\sum_{j=d-k+1}^d\la^{r_2-(d-1)}_{\hat\pi_0^{-1}(j)}\Big)$. Hence 
	\[
	I^2=\sum_{j=1}^d\la^{r_2-(d-1)}_{\hat\pi_0^{-1}(j)}-\sum_{j=2}^d\la^{r_2-(d-1)}_{\hat\pi_0^{-1}(j)}=\la^{r_2-(d-1)}_{\hat\pi_0^{-1}(1)}=Q_2(\la^{r_2-(d-1)}).
	\]
	
	Moreover, note that the permutation $\hat\pi^1$ obtained after one step of Rauzy Veech induction is given by
	\[
	\begin{split}
	\hat\pi^1_0(\hat\pi_0^{-1}(1))&=1;\\
	\hat\pi^1_0(\hat\pi_0^{-1}(d))&=2=d-(d-1)+1;\\
	\hat\pi^1_0(\hat\pi_0^{-1}(j))&=j+1\quad\text{for}\quad j=2,\ldots,d-1;\\
	\hat\pi^1_1&=\hat\pi_1.
	\end{split}
	\]
	Furthermore, for $k=2,\ldots,d-1$ by induction we get
		\[
		\begin{split}
		\hat\pi^k_0(\hat\pi_0^{-1}(1))&=1\\
		\hat\pi^k_0(\hat\pi_0^{-1}(j))&=2=j-(d-1)+k\quad\text{for}\quad j=d-k+1,\ldots,d\\
		\hat\pi^k_0(\hat\pi_0^{-1}(j))&=j+k\quad\text{for}\quad j=2,\ldots,d-k\\
		\hat\pi^k_1&=\hat\pi_1
		\end{split}
		\]
	By considering $k=d-1$ we get $\hat\pi^{d-1}=\hat\pi$. 
	
	Moreover, since by performing $(d-1)$ steps of Rauzy-Veech induction we only ,,cut'' from the interval corresponding to the symbol $\hat\pi_0^{-1}(1)$, we get that
	\[
	\la^{r_2}_{\hat\pi_0^{-1}(1)}=\la^{r_2-(d-1)}_{\hat\pi_0^{-1}(1)}-\sum_{j=2}^d\la^{r_2-(d-1)}_{\hat\pi_0^{-1}(j)}\quad\text{and}\quad\la^{r_2}_{\hat\pi_0^{-1}(j)}=\la^{r_2-(d-1)}_{\hat\pi_0^{-1}(j)}\quad \text{for}\ j=2,\ldots,d.
	\]
	
	Finally, since for every $j=2,\ldots,d$, we have $I_{\hat\pi_0^{-1}(j)}^2\subset I_{\hat\pi_0^{-1}(1)}^{r_n-(d-1)}$ and $T^{s_{\hat\pi_0^{-1}(1)}^2}I_{\hat\pi_0^{-1}(j)}^2=I_{\hat\pi_0^{-1}(j)}^{r_n-(d-1)}$, we get that the first return time for the interval $I_{\pi_0^{-1}(j)}^2$ to $I^2$ (and thus the height of the corresponding tower) equals $s_1^2+s_j^2$. Moreover the first return time of the interval $I_{\pi_0^{-1}(1)}^{2}$ to $I^2$ does not change in comparison to the first return time of the interval $I_{\hat\pi_0^{-1}(1)}^{r_2-(d-1)}$ to $I^{r_2-(d-1)}$ and thus is equal to $s_1^2$.

	\vspace{2mm}
	\noindent To summarize we obtain a decomposition of an interval into one short tower of width a little smaller then the half of length of the domain, $d-2$ very high but very thin towers, and one very high tower of width a little greater then the half of length of the domain (see Figure \ref{rys1}). The reader should have in mind that this is made so that the interval exchange transformation under consideration resembles a rotation with bounded partial quotients along subsequence. This construction leads to different heights of the first and the last tower (with all others negligible) is crucial to find a sequence which is a rigidity sequence along a family of subsets whose measure converges to 1. It is emphasised in the continuation of the proof.}

	\subsection{The description of the construction.}\label{sec:const} 
%	\textit{Towers over $I^{n}_{\hat\pi_0^{-1}(j)}$ for $j=2,\ldots,d-1$.} 
	Observe that in view of \eqref{ineq1}, \eqref{ineq2}, \eqref{eq:distest} and the fact that $\la_{\hat\pi_0^{-1}(1)}^{r_n}>0$ we obtain that
	\begin{equation}\label{ineq5}
	Q_n(\la^{r_n-(d-1)(n-1)})<2\la^{r_n}_{\hat\pi_0^{-1}(d)}<\frac{2}{n-1}\la^{r_n-(n-1)(d-1)}_{\hat\pi_0^{-1}(1)}.
	\end{equation}

	We now prove that the measure of the union of towers over intervals $I^{n}_{\hat\pi_0^{-1}(j)}$ for $j=2,\ldots,d-1$ tends to $0$ as $n\to\infty$. Indeed, by \eqref{ineq3} and \eqref{ineq5} we have
	\be\label{eq:IETmes}
	\begin{split}
		Leb&\big(\bigcup_{j=2,\ldots,d-1}\bigcup_{i=0}^{s_j^n+(n-1)s_1^n-1}T^iI^n_{\hat\pi_0^{-1}(j)}\big)
		=\sum_{j=2}^{d-1}\la_{\hat\pi_0^{-1}(j)}^{r_n}(s_j^n+(n-1)s_1^n)\\
		&<\sum_{j=2}^{d-1}\la_{\hat\pi_0^{-1}(j)}^{r_n}s_1^n(n-1+\rho(B))<\frac{1}{n+\rho(B)}Q_n(\la^{r_n-(d-1)(n-1)})s_1^{n}(n-1+\rho(B))\\
		&<\frac{2}{n-1}\la^{r_n-(n-1)(d-1)}_{\hat\pi_0^{-1}(1)}s_1^{n}<\frac{2}{n-1}\to 0\ \text{as}\ n\to\infty,
	\end{split}
	\ee
	where in the last inequality we used the fact that the interval $[0,\la^{r_n-(n-1)(d-1)}_{\hat\pi_0^{-1}(1)})$ is a base of Rokhlin tower of height $s_1^n$.
	
	\textit{Discontinuities of $T$.} Recall that by the definition of IET the discontinuities of $T$ are leftpoints of intervals $I_{a}$ for $a\in\mathcal A$ . Denote those points by $\partial I_a$. By Remark \ref{rozkladnawieze} we obtain that for each $a\in\mathcal A$, the point $\partial I_a$ is the leftpoint of some level of tower $\{T^{i}I^n_{\hat\pi_0^{-1}(j)};\ 0\le i<s_j^{n}+(n-1)s_1^n\}$ for some $j=2,\ldots,d$. However, note that all towers $\{T^{i}I^n_{\hat\pi_0^{-1}(j)};\ 0\le i<(n-1)s_1^n\}$ are contained in the tower over $[0,\la^{r_n-(n-1)(d-1)}_{\hat\pi_0^{-1}(1)})$ of height $s_1^n$. Hence those towers \underline{cannot} contain discontinuities of $T$. In fact, for all $s_1^n\le i<(n-2)s_1^n$ the sets
	
	\begin{equation}\label{glue0}
	\begin{split} &\text{$T^i[\la^{r_n}_{\hat\pi_0^{-1}(1)},Q_n(\la^{r_n-(d-1)(n-1)}))$, $T^{i-s_1^n}[\la^{r_n}_{\hat\pi_0^{-1}(1)},Q_n(\la^{r_n-(d-1)(n-1)}))$},\\ &\quad\text{and $T^{i+s_1^n}[\la^{r_n}_{\hat\pi_0^{-1}(1)},Q_n(\la^{r_n-(d-1)(n-1)}))$}\\&
	\text{are intervals included in $T^{i\ mod\ s_1^n}[0,\la^{r_n-(n-1)(d-1)}_{\hat\pi_0^{-1}(1)})$.}
	\end{split}
	\end{equation} Moreover 
	\begin{equation}\label{glue}
	\begin{split}
	&\text{the leftpoint of $T^i[\la^{r_n}_{\hat\pi_0^{-1}(1)},Q_n(\la^{r_n-(d-1)(n-1)}))$}\\&\quad  \text{is the rightpoint of $T^{i-s_1^n}[\la^{r_n}_{\hat\pi_0^{-1}(1)},Q_n(\la^{r_n-(d-1)(n-1)}))$,}
	\\\ &\text{while the rightpoint of $T^i[\la^{r_n}_{\hat\pi_0^{-1}(1)},Q_n(\la^{r_n-(d-1)(n-1)}))$}\\&\quad \text{is the leftpoint of  $T^{i+s_1^n}[\la^{r_n}_{\hat\pi_0^{-1}(1)},Q_n(\la^{r_n-(d-1)(n-1)}))$.}
	\end{split}
	\end{equation}
	In particular all discontinuities of $T$ are contained in towers of form $\{T^{i}I^n_{\hat\pi_0^{-1}(j)};\ (n-1)s_1^n\le i<s_j^{n}+(n-1)s_1^n\}$ for some $j=2,\ldots,d$.
	
	\textit{Tower $W_n$.} For any $n\in\N$ define a tower $W_n:=\bigcup_{i=0}^{q_n-1} J^n$, where
	\begin{itemize}
		\item $q_n=s_d^n+ns_1^n$;\\
		\item $J^n:=I_{\hat\pi_0^{-1}(1)}^n\cap T^{-q_n}I_{\hat\pi_0^{-1}(1)}^n\cap\ldots\cap T^{-Kq_n}I_{\hat\pi_0^{-1}(1)}^n.$
	\end{itemize}
	Denote
	\[
	\Delta_n:=\la^{r_n}_{\hat\pi_0^{-1}(d)}-\la^{r_n}_{\hat\pi_0^{-1}(1)}>0,
	\]
	where the inequality follows from \eqref{ineq4}.
	Note that $T^{s_1^n}(I_{\hat\pi_0^{-1}(1)}^n)$ is an interval, which is the rightmost part of the interval $I_{\hat\pi_0^{-1}(d)}^n$ and thus 
	\begin{equation}\label{gdzieI_1}
	T^{q_n}(I_{\hat\pi_0^{-1}(1)}^n)=T^{s_d^n+(n-1)s_1^n}(T^{s_1^n}I_{\hat\pi_0^{-1}(1)}^n)=\big[\Delta_n,\la^{r_n}_{\hat\pi_0^{-1}(1)}+\Delta_n\big). 
	\end{equation}
	In particular $I_{\hat\pi_0^{-1}(1)}^n\cap T^{q_n}I_{\hat\pi_0^{-1}(1)}^n$ is an interval of length $\la^{r_n}_{\hat\pi_0^{-1}(1)}-\Delta_n$. Since it is an subinterval of $T^{s_d^n+(n-1)s_1^n}I^{r_n}_{\hat\pi_0^{-1}(d)}$ on which $T^{-q_n}$ acts by translation by $-\Delta_n$, we get that 
	$I_{\hat\pi_0^{-1}(1)}^n\cap T^{-q_n}I_{\hat\pi_0^{-1}(1)}^n=\big[0,\la^{r_n}_{\hat\pi_0^{-1}(1)}-\Delta_n\big)$  Note also that in view of \eqref{ineq4} {(see also \eqref{eq:tilte})}, \eqref{ineq1} {and the defintion of $\ep$,} we get 
	\begin{equation}\label{3/5}
	K\Delta_n<K\ep Q_n(\la^{r_n-(d-1)(n-1)})<\frac{3K}{100K}\la^{r_n}_{\hat\pi_0^{-1}(1)}=\frac{3}{100}\la^{r_n}_{\hat\pi_0^{-1}(1)}.
	\end{equation}
	Thus, by induction, we get that
	\[
	J^n=\big[0,\la^{r_n}_{\hat\pi_0^{-1}(1)}-K\Delta_n\big)
	\]
	is a non-empty interval.
	
	Moreover, since by \eqref{ineq1} and the choice of $\ep$ we get that $\la^{r_n}_{\hat\pi_0^{-1}(1)}>\frac{1}{3}|I^n|$, in view of \eqref{3/5}, we get 
	\begin{equation}\label{Jdown}
	Leb(J^n)=\la^{r_n}_{\hat\pi_0^{-1}(1)}-K\Delta_n>
	\frac{97}{100}\la^{r_n}_{\hat\pi_0^{-1}(1)}>\frac{3}{10}|I^n|=\frac{3}{10}\sum_{j=1}^d\la^{r_n}_{\hat\pi_0^{-1}(j)},
	\end{equation}
	which implies that
	\[
	\begin{split}
	Leb(W_n)&=q_n\cdot Leb(J^n)>\frac{3}{10}(s_d^n+ns_1^n)\sum_{j=1}^d\la^{r_n}_{\hat\pi_0^{-1}(j)}\\
	&>\frac{3}{10\rho(B)}\big( 
	s_1^n\la^{r_n}_{\hat\pi_0^{-1}(1)}+\sum_{j=2}^d(s_j^n+(n-1)s_1^n)\la^{r_n}_{\hat\pi_0^{-1}(j)} \big)=\frac{3}{10\rho(B)},
	\end{split}
	\]
	where in the last equation we used the fact that the union of all considered towers over intervals $I^n_{\hat\pi_0^{-1}(j)}$ is the unit interval. In particular we obtained that
	\begin{equation}
	\liminf_{n\to\infty}Leb(W_n)> 0.
	\end{equation}
	
	We will also define here a number which will play crucial role in the conclusion of the proof. Note that also by \eqref{ineq4} {and \eqref{eq:distest},} we have
	\begin{equation}\label{3/10}
	\begin{split}
	q_n\Delta_n
	&>\frac{\ep}{2}q_n Q_n(\la^{r_n-(d-1)(n-1)})\\
	&>\frac{\ep}{2\rho(B)}\big(s_1^n\la_{\hat\pi_0^{-1}(1)}^{r_n}+\sum_{i=2}^d(s_i^n+(n-1)s_1^n)\la_{\hat\pi_0^{-1}(i)}^{r_n}\big)=\frac{\ep}{2\rho(B)}>0.
		\end{split}
		\end{equation}
	In particular by passing to a subsequence if necessary and also by using the fact $\lim_{n\to\infty}\frac{(n-1)s_1^n}{q_n}=1$ we obtain the there exists $\g>0$ such that
	\begin{equation}\label{qnlan}
	\lim_{n\to\infty}q_n\Delta_n=
	\lim_{n\to\infty}(n-1)s_1^n\Delta_n=\g
	\end{equation}
	
	\textit{Set of non-controlled points}. For each $n\in\N$ we now describe the set $X_n$ of points for which we do not control the values of cocycle $f$. It consists of three parts. First are the towers $\{T^{i}I^{n}_{\hat\pi_0^{-1}(j)};\ 0\le i<s_j^n+(n-1)s_1^n \}$ for $2\le j\le d$. We have already proved though {(see eg. \eqref{eq:IETmes})}, that the measure of union of those towers converges to $0$ as $n\to\infty$. The second part is the tower 
	\[
	\{T^{i}[|I^n|-K\Delta_n,|I^n|);\ (n-3)s_1^n\le i<s_d^n+(n-1)s_1^n\}.
	\]	
	However, in view of \eqref{ineq4} and \eqref{ineq2}, the measure of this set is  equal to 
	\[
	\begin{split}
	K&\Delta_n(2s_1^n+s_d^n)<(2+\rho(B))s_1^nK\ep|I^{n}|\\
	&\le2(2+\rho(B))\frac{s_d^{n}+(n-1)s_1^{n}}{n-1}K\ep|I^{n}_{\hat\pi_0^{-1}(d)}|\le\frac{2(2+\rho(B))}{n-1}K\ep\to 0\text{ as }n\to\infty.
	\end{split}
	\]
	The final part of $X_n$ is the set $\bigcup_{i=0}^{Kq_n-1}T^{-i}(\bigcup_{j=2}^{d-1}I_{\hat\pi_0^{-1}(j)}^n)$. By \eqref{ineq3} its measure however is bounded by
	\[
	\begin{split}
	Kq_n\sum_{j=2}^{d-1}\la^{r_n}_{\hat\pi_0^{-1}(j)}&<K(n+\rho(B))s_1^{r_n}\frac{|I^n|}{n+\rho(B)}\\&<4K\la^{r_n}_{\hat\pi_0^{-1}(d)}\frac{s_d^n+(n-1)s_1^n}{n-1}<\frac{4K}{n-1}\to 0\text{ as }n\to\infty.
	\end{split}
	\]
	To sum up we obtained that sets $X_n$ consisting of three parts listed above satisfy
	\begin{equation}\label{miaraXn}
	\lim_{n\to\infty}Leb(X_n)=0.
	\end{equation}
	
	\begin{figure}[h]
		\includegraphics[scale=0.7]{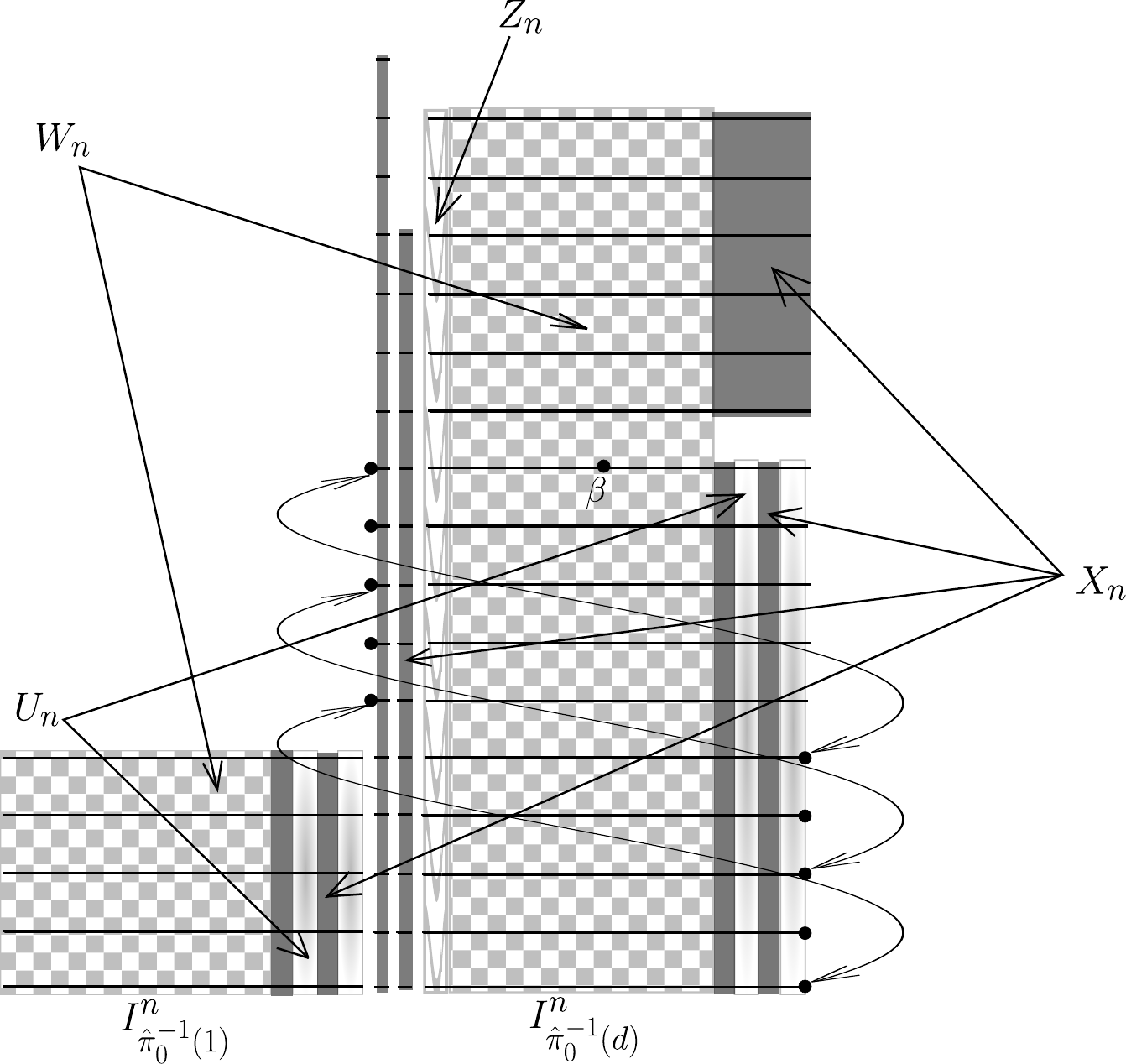}
		$\quad\quad\quad\quad$
		\caption{Sets $W_n$, $Z_n$, $U_n$ and $X_n$ obtained in the construction. Curved arrows show which points are identified in view of \eqref{glue}.\label{rys2}}
	\end{figure}
	
	\medskip
	\textit{Set of controlled points.} We now give precise description of points which are included in the set $I\setminus X_n$. Namely we show that $I\setminus X_n$ is a union of three disjoint sets $W_n,Z_n$ and $U_n$, where $W_n$ is defined as before. 
	
	For every $n\in\N$ denote $\Sigma_n:=\sum_{j\le d-1}\la^{r_n}_{\hat\pi_0^{-1}(j)}$. Moreover, for each $n\in\N$ let 
	\[
	Z_n:=\bigcup_{i=0}^{s_d^n+(n-1)s_1^n-1}T^i\big[\Sigma_n,\Sigma_n+\Delta_n \big)
	\]
	and
	\[
	U_n:=\bigcup_{i=0}^{(n-2)s_1^n-1}T^i\big [|I^{n}_{\hat\pi_0^{-1}(1)}|-K\Delta_n,|I^{n}_{\hat\pi_0^{-1}(1)}|\big)\setminus X_n,
	\]
	see Figure \ref{rys2}.
	The set $Z_n$ is the leftmost part of Rokhlin tower over $I^n_{\hat\pi_0^{-1}(d)}$ of height $s_d^n+(n-1)s_1^n$. The set $U_n$ on the other hand is the rightmost part of the tower $\{T^iI^n_{\hat\pi_0^{-1}(1)};\ 0\le i <q_n\}$, excluding a portion of top levels and preimages of intervals $I^n_{\hat\pi_0^{-1}(j)}$ for $j=2,\ldots,d-1$.
	More precisely, by \eqref{miaraXn} and \eqref{qnlan} we have that there exists a sequence $\{\de_n\}_{n\in\N}$ of positive numbers such that
	\begin{equation}\label{pdec}
		\lim_{n\to\infty}\frac{\de_n}{\Delta_n}=0\ \text{and}\ T^{-m}(x)\in\big[|J^n|+(p-1)\Delta_n+\de_n,|J^n|+p\Delta_n \big).
	\end{equation}
	Each of the set $W_n$, $Z_n$ and $U_n$ is a nontrivial part of $I$. We will now show that these are indeed sets on which we control the asymptotic behaviour of $T$.
	
	\medskip	
	\subsection{The rigidity time.}\label{sec:rig} We now show that for any $i=1,\ldots, K$, the sequence  $\{iq_n\}_{n\in\N}$ is a rigidity time along the sequence $\{I\setminus X_n\}_{n\in\N}$. It is worth to mention that since $Leb(X_n)\to 0$, we actually obtain that each of those sequence is a regular rigidity time for $T$. Hence it should be enough that  $\{q_n\}_{n\in\N}$ is a rigidity sequence. For clarity and to point out certain phenomena we give the proof for every $i=1,\ldots,K$. To do this we want to use Lemma \ref{metr}. By \eqref{gdzieI_1} we have that 
	\begin{equation}\label{transWn}
	T^{iq_n}x=x+i\Delta_n\ \text{for}\ x\in W_n.
	\end{equation}
	Hence
	\begin{equation}
	\lim_{n\to\infty}|T^{iq_n}x-x|=0\ \text{for}\ x\in W_n.
	\end{equation}
	We will now show that the above convergence is also true for $x\in Z_n\cup U_n$. 
	Let $x\in Z_n$. Then for some $l=0,\ldots,s_d^n+(n-1)s_1^n$ we have
	\[
	x\in T^l\big[\Sigma_n,\Sigma_n+\Delta_n \big)\subset T^lI^n_{\hat\pi_0^{-1}(d)}.
	\]
	Recall that 
	\begin{equation}\label{gdzieId}
	T^{s_d^n+(n-1)s_1^n}I^n_{\hat\pi_0^{-1}(d)}=[0,|I^n_{\hat\pi_0^{-1}(d)}|).
	\end{equation}
	Moreover in view of \eqref{ineq4} and \eqref{ineq1} we have
	\[
	2\Delta_n<2\ep|I^n|<\frac{1}{50}|I^n|<|I^n_{\hat\pi_0^{-1}(1)}|.
	\]
	Thus
	\[
	T^{s_d^n+(n-1)s_1^n}\big[\Sigma_n,\Sigma_n+\Delta_n\big)=[0,\Delta_n)\subset I^n_{\hat\pi_0^{-1}(1)}.
	\]
	Since $T^{s_1^n}I^n_{\hat\pi_0^{-1}(1)}\subset I^n_d$, we get
	\begin{equation}
	T^{q_n}x=x+\Delta_n\subset T^l I^n_{\hat\pi_0^{-1}(d)}.
	\end{equation}
	More precisely, 
	\begin{equation}
	\begin{split}\label{ZnwWn}
	T^{q_n}(x)&\in T^{l}\big[\Sigma_n+\Delta_n,\Sigma_n+2\Delta_n\big)\\&\subset T^{s_1^n+l}J^n\subset W_n.
	\end{split}
	\end{equation}
	Hence by \eqref{transWn} we obtain that
	\begin{equation}\label{transZn}
	T^{iq_n}x=T^{(i-1)q_n}(T^{q_n}x)=x+i\Delta_n\text{ for }x\in Z_n.
	\end{equation}
	Hence for every $i=1,\ldots,K$ it holds that
	\begin{equation}
	\lim_{n\to\infty}\sup_{x\in Z_n}|T^{iq_n}x-x|=0
	\end{equation}
	
	Assume now that $x\in U_n$ and for $j=1,\ldots,K$ define 
	\[
	\begin{split}
	U_{n,j}:=\bigcup_{i=0}^{(n-2)s_1^n-1}T^i\big[|I^{n}_{\hat\pi_0^{-1}(1)}|-&(K-j+1)\Delta_n,\\
	&\quad|I^{n}_{\hat\pi_0^{-1}(1)}|-(K-j)\Delta_n\big)\setminus X_n.
	\end{split}
	\]
	Then for some $1\le l\le K$ we have $x\in U_{n,l}$. That is there exists $0\le m<(n-2)s_1^n$ such that 
	\[
	\begin{split}
	x&\in T^m\big[|I^{n}_{\hat\pi_0^{-1}(1)}|-(K-l+1)\Delta_n,|I^{n}_{\hat\pi_0^{-1}(1)}|-(K-l)\Delta_n\big)\\
	&=T^m\big[|J^n|+(l-1)\Delta_n,|J^n|+l\Delta_n\big).
	\end{split}
	\]
	
	Now assume that $i=1,\ldots,K-l$. By \eqref{gdzieI_1} and the fact that $T$ acts on each level of towers $\{T^MI^n_{\hat\pi_0^{-1}(1)};\ 0\le M<s_1^n \}$ and $\{T^MI^n_{\hat\pi_0^{-1}(d)};\ 0\le M<q_n-s_1^n \}$ via translation, we obtain
	\[
	T^{q_n-m}x=T^{q_n}(T^{-m}x)=T^{-m}x+\Delta_n\in \big[|J^n|+l\Delta_n,|J^n|+(l+1)\Delta_n\big)\subset I^n_{\hat\pi_0^{-1}(1)},
	\]
	and thus
	\[
	T^{q_n}x=x+\Delta_n\in T^m \big[|J^n|+l\Delta_n,|J^n|+(l+1)\Delta_n\big)\subset T^mI^n_{\hat\pi_0^{-1}(1)}. 
	\]
	Analogously, for every $i=1,\ldots,K-l$ we obtain	
	\begin{equation}\label{transUn1}
	T^{iq_n}x=x+i\Delta_n.
	\end{equation}
	and
	\begin{equation}\label{transUn2}
	\begin{split}
	T^{iq_n}x\in T^m \big[&|J^n|+(l+i-1)\Delta_n,\\
	&|J^n|+(l+i)\Delta_n\big)\subset T^mI^n_{\hat\pi_0^{-1}(1)} 
	\end{split}
	\end{equation}
	Thus for $i=1,\ldots,K-l$ we obtained that 
	\[
	\lim_{n\to\infty}\sup_{x\in U_{n,l}}|T^{iq_n}x-x|=0.
	\]
	Consider now $i=K-l+1,\ldots,K$. We have already proved that
	\[
	T^{(K-l)q_n}x\in T^m \big[|J^n|+(K-1)\Delta_n,|J^n|+K\Delta_n\big)\setminus X_n\subset U_{n,K}.
	\]
	In view of \eqref{gdzieId} and the fact that $x\notin X_n$ we get that
	\[
	T^{(K-l+1)q_n-m}x\in \big[\Sigma_n,\Sigma_n+\Delta_n \big)
	\] 
	and thus
	\begin{equation}\label{przeskok1}
	T^{(K-l+1)q_n}x\in T^m\big[\Sigma_n,\Sigma_n+\Delta_n \big).
	\end{equation}
	If $m<s_1^n$ then $T^{(K-l)q_n}x,T^{(K-l+1)q_n}x\in T^m{I^n}$ and in view of \eqref{gdzieI_1} and \eqref{glue} we have
	\begin{equation}\label{przeskok}
	\begin{split}
	T^{(K-l+1)q_n}x&=T^{m}(T^{(K-l+1)q_n-m}x)\\
	&=T^m(T^{(K-l)q_n-m}x+\Delta_n) =T^{(K-l)q_n}x+\Delta_n.
	\end{split}
	\end{equation}
	For $m\ge s_1^n$, again by using \eqref{glue} we get for points $T^{(K-l)q_n}x\in T^{m-s_1^n}I^n_{\hat\pi_0^{-1}(d)}$ and $T^{(K-l+1)q_n}x\in T^{m}I^n_{\hat\pi_0^{-1}(d)}$ and moreover $T^{(K-l+1)q_n}x=T^{(K-l)q_n}x+\Delta_n$, that is property \eqref{przeskok} holds for every $x\in U_{n,l}$. Since in view of \eqref{przeskok1} we have that
	\begin{equation}\label{przeskok2}
	T^{(K-l+1)q_n}x\in Z_n,
	\end{equation}
	 by \eqref{transZn} we obtain that for any $i=K-l,\ldots,K$ the property \eqref{transUn1} holds. To sum up we get that for any $i=1,\ldots,K$ we have
	\begin{equation}\label{transUn}
	T^{iq_n}x=x+i\Delta_n\quad\text{for every }x\in U_n
	\end{equation}
	and thus
	\[
	\lim_{n\to\infty}\sup_{x\in U_n}|T^{iq_n}x-x|=0.
	\]
	This concludes the proof that $iq_n$ is a rigidity time along the sequence $\{W_n\cup Z_n\cup U_n\}_{n\in\N}$ for every $i=1,\ldots,K$.
	\subsection{Piecewise constant roof functions.} All objects defined in sections \ref{sec:rec}-\ref{sec:rig} are used in proving both parts of Theorem \ref{Main}. From this point we evaluate the values of the cocycle for piecewise constant and piecewise linear with non-zero slope roof functions separately. On this basis we calculate the limit measure which differentiates between the powers of special flows under consideration. 
	\subsubsection{\bf Roof function.} Let $D_{\hat\pi,\la}\subset[0,1)$ be such that for every $\beta\in D_{\hat\pi,\la}$ we have that
	\[
	\beta\in\bigcup_{i=0}^{q_n-1}T^i\big[\frac{1}{2}|J^n|,|J^n|)\big)\ \text{ for infinitely many }n\in\N.
	\]
	Then for those $n\in\N$ and $0\le i<K$ we have
	\[
	\bigcup_{j=iq_n}^{(i+1)q_n-1}T^{-j}[\beta-\Delta_n,\beta)\subset W_n.
	\]
	Indeed, in view of \eqref{3/5} and \eqref{Jdown} we get that 
	\begin{equation}\label{1/10}
		(K+1)\Delta_n\le 2K\Delta_n<\frac{3}{50}\la^{r_n}_{\hat\pi_0^{-1}(1)}<\frac{1}{10}|J^n|,
	\end{equation}
	from which the above inclusion follows.
	
	From now on we assume that $f$ is a piecewise constant function, which has discontinuity point in $\beta$ and all other discontinuities are included in the set of discontinuity points of $T$. We denote the jump of $f$ in $\beta$ by $D_\beta$. 
	
	\textit{The values of the cocycle.} We now prove that there exists a real sequence $\{a_n\}_{n\in\N}$ such that for every $i=1,\ldots,K$ we have that
	\begin{equation}\label{jestemogr}
	\text{the sequences }\big\{\int_{W_n\cup Z_n\cup U_n}|S_{iq_n}(f)(x)-ia_n|^2\,dx\big\}_{n\in\N}\ \text{ are bounded.}
	\end{equation}
	Let $a_n:=S_{q_n}(f)(0)$. For every $x\in W_n\cup Z_n\cup U_n$ we evaluate the difference $S_{iq_n}(f)(x)-ia_n$ for $i=1,\ldots,K$. Since these values will be uniformly bounded, \eqref{jestemogr} will be straightforward. 
	
	\textit{Points from $W_n$.} Let $i=1,\ldots,K$. Assume first that $x\in W_n$. Then for some $l=0,\ldots,q_n-1$ we have $x\in T^lJ^n$. Then one of the following possibilities holds:
\begin{enumerate}
		\item[\emph{(i)}] the orbit $\{T^kx\}_{k=0,\ldots,iq_n-1}$ intersects the interval $T^lI_{\hat\pi_0^{-1}(1)}$ only to the left of $\beta$. Then \begin{equation}\label{limitWn-2}
			S_{iq_n}(f)(x)=iS_{q_n}(f)(0).
			\end{equation}
		\item[\emph{(ii)}] the orbit $\{T^kx\}_{k=0,\ldots,iq_n-1}$ intersects the interval $T^lI_{\hat\pi_0^{-1}(1)}$ only to the right of $\beta$. Then 
		\begin{equation}\label{limitWn-1}
		S_{iq_n}(f)(x)=iS_{q_n}(f)(0)+iD_\beta.
		\end{equation}
		\item[\emph{(iii)}] $i\ge 2$ and for some $k=1,\ldots,i-1$ we have 
		\[
		x\in \bigcup_{j=(i-k-1)q_n}^{(i-k)q_n-1}T^{-j}[\beta-\Delta_n,\beta)\subset W_n.
		\]
	\noindent Then 
	\begin{equation}\label{limitWn}
		S_{iq_n}(f)(x)=iS_{q_n}(f)(0)+kD_\beta.
		\end{equation}
	\end{enumerate}
	The points $(i)$ and $(ii)$ follow from the fact that $f$ is constant on all level intervals of the tower $\{T^M I_{\hat\pi_0^{-1}(1)}^n;\ 0\le M<q_n\}$, except the level to which $\beta$ belongs. Since the orbit $\{T^j(0)\}_{j=0}^{q_n-1}$ is precisely the set of left endpoints of aforementioned tower and by the definition of $W_n$ each orbit segment of length $q_n$ intersects each level of this tower exactly once, the points $(i)$ and $(ii)$ follow.
	
	Assume now that $i\ge 2$ and for some $k=1,\ldots,i-1$ we have 
	\begin{equation}\label{inne}
	x\in \bigcup_{j=(i-k-1)q_n}^{(i-k)q_n-1}T^{-j}[\beta-\Delta_n,\beta).
	\end{equation}
	Then for $m=0,\ldots,i-k-1$ the orbit segment $\{T^jx\}_{j=mq_n}^{(m+1)q_n}$ intersects the interval $T^l I_{\hat\pi_0^{-1}(1)}$ to the left of $\beta$. Hence  
	\begin{equation}\label{wn1}
	S_{q_n}(f)(T^{mq_n}x)=S_{q_n}(f)(0)\ \text{ for }\ m=0,\ldots,i-k-1.
	\end{equation}
	On the other hand, for $m=i-k,\ldots,i-1$ the orbit segment $\{T^jx\}_{j=mq_n}^{(m+1)q_n}$ intersects the interval $T^l I_{\hat\pi_0^{-1}(1)}$ to the right of $\beta$. Hence
	\begin{equation}\label{wn2}
	S_{q_n}(f)(T^{mq_n}x)=S_{q_n}(f)(0)+D_\beta\ \text{ for }\ m=i-k,\ldots,i-1.
	\end{equation}
	To summarize, \eqref{wn1} and \eqref{wn2} together yield
	\begin{equation}
	S_{iq_n}(f)(x)=\sum_{m=0}^{i-1}S_{q_n}(f)(T^{mq_n}x)=iS_{q_n}(f)(0)+kD_\beta\ \text{for}\ x\ \text{satisfying \eqref{inne}}.
	\end{equation}

	\textit{Points from $Z_n$.}
	Assume now that $x\in Z_n$, that is 
	for some $m=0,\ldots,s_d^n+(n-1)s_1^n-1$ we have
	\[
	x\in T^m\big[\Sigma_n,\Sigma_n+\Delta_n\big)\subset T^mI^n_{\hat\pi_0^{-1}(d)}.
	\]
	Then by \eqref{transZn} for $0\le j<q_n$ and $b=0,\ldots,i-1$ we have
	\begin{equation}\label{Znupper}
	T^{bq_n}(T^jx)=T^jx+b\Delta_n.
	\end{equation}
	Moreover note that since $\beta\in T^lJ^n$ we have that if $l<s_1^n$ then $\beta\in T^lI_{\hat\pi_0^{-1}(1)}^n$ and if $l\ge s_1^n$ then $\beta\in T^{l-s_1^n}I_{\hat\pi_0^{-1}(d)}^n$ . Additionally
	\begin{equation}\label{Zndist}
	T^jx\in T^{m+j}I^n_{\hat\pi_0^{-1}(d)}\ \text{ for }\ 0\le j< q_n-m-1 .
	\end{equation}
	Then by \eqref{Znupper}, \eqref{1/10} and the definition of $Z_n$ the orbit $\{T^jx\}_{j=0}^{iq_n-1}$ is contained in tower 
	\begin{equation}\label{Zntower}
	\big\{\bigcup_{j=0}^{q_n-1}T^j\big[\Sigma_n,\Sigma_n+\frac{1}{10}|J^n| \big\}
	\end{equation}
	Since $\beta\in T^l[\frac{1}{2}|J^n|,|J^n|)$, the form of above tower implies that we cannot have point $\beta$ to the left of orbit $\{T^jx\}_{j=0}^{iq_n-1}$. This implies that 
	\begin{equation}\label{limitZn}
	S_{iq_n}(f)(x)=S_{iq_n}(f)(0)\ \text{for}\ x\in Z_n\ \text{and}\ i=1,\ldots,K.
	\end{equation}
	
	\textit{Points from $U_n$.} Suppose now that $x\in U_n$ that is $x\in U_{n,p}$ for some $1\le p\le K$. Then there exists $0\le b<(n-2)s_1^n$ such that 
	\[
	x\in T^b \big[|J^n|+(p-1)\Delta_n,|J^n|+p\Delta_n \big)\subset T^bI_{\hat\pi_0^{-1}(1)}^n.
	\]
	If $p<K$ then for $m=0,\ldots,K-p-1$ we have that the orbit segments $\{T^jx\}_{j=mq_n}^{(m+1)q_n}$ intersect each level of tower $\{T^MI^n_{\hat\pi_0^{-1}(1)};\ 0\le M<s_1^n \}$ and $\{T^MI^n_{\hat\pi_0^{-1}(d)};\ 0\le M<q_n-s_1^n \}$ exactly once. 
	Moreover whenever they intersect the interval $T^l I_{\hat\pi_0^{-1}(1)}$ they do it to the right of $\beta$. Hence for $m=0,\ldots,K-p-1$ we have that 
	\[
	S_{q_n}(f)(T^{mq_n}x)=S_{q_n}(f)(0)+D_\beta
	\]
	and thus if $1\le p\le K-i$ then
	\begin{equation}\label{wn01}
	S_{iq_n}(f)(x)=if^{(q_n)}(x)+iD_\beta\ \text{ for }\ x\in U_{n,p}.
	\end{equation}
	
	By the definition of $U_{n,p}$ we have that 
	\begin{equation}\label{slip}
	T^{(K-p)q_n+q_n-m}x\in[\Sigma_n,\Sigma_n+\Delta_n \big)
	\end{equation}
	Recall that $T$ acts on each level of towers $\{T^MI^n_{\hat\pi_0^{-1}(c)};\ 0\le M<s_c^n \}$ by translation. Then in view of \eqref{glue0} and \eqref{glue}, for every $j=0,\ldots,q_n-1$ the set $T^j[T^{(K-p-1)q_n}x,T^{(K-p)q_n}x]$ is an \underline{interval} of length $\Delta_n$ which does not contain $\beta$. Thus
	\begin{equation}\label{wn02}
	S_{q_n}(f)(T^{(K-p)q_n}x)=S_{q_n}(f)(T^{(K-p-1)q_n}x)=S_{q_n}(f)(0)+D_\beta
	\end{equation}
	Finally note that $T^{(K-p+1)q_n}x\in Z_n$. Thus if $p>1$ then in view of \eqref{limitZn} we obtain for every $m=K-p+1,\ldots,K-1$ the following
	\begin{equation}\label{wn03}
	S_{q_n}(f)(T^{mq_n}x)=S_{q_n}(f)(0)\ \text{for every}\ x\in U_{n,p}.
	\end{equation}
	By combining \eqref{wn01}, \eqref{wn02} and \eqref{wn03} if $K-i+1\le p\le K$ ($K-p+1\le i \le K$) then
	\begin{equation}\label{wn04}
	S_{iq_n}(f)(x)=iS_{q_n}(f)(0)+(K-p+1)D_\beta\ \text{ for every }\ x\in U_{n,p}.
	\end{equation}
	
	\subsubsection{\bf{The evaluation of $S_{iq_n}(f)$.}} We now summarize the results of three previous parts of the proof and obtain the values of $S_{iq_n}(f)$ on $W_n\cup Z_n\cup U_n$ for every $i=1,\ldots,K$. By \eqref{limitWn-2}, \eqref{limitWn-1}, \eqref{limitWn}, \eqref{limitZn}, \eqref{wn01} and \eqref{wn04} for every $x\in W_n\cup Z_n\cup U_n$ we have
	\[
	S_{iq_n}(f)(x)-iS_{q_n}(f)(0)=jD_\beta\quad\text{for some}\ j=0,\ldots,i.
	\]
	More precisely, if $0<j<i$ then for $x\in W_n\cup Z_n\cup U_n$ we have
	\[
	S_{iq_n}(f)(x)-iS_{q_n}(f)(0)=jD_\beta\Leftrightarrow x\in\bigcup_{l=(j-k+1)q_n}^{(j-k)q_n-1}T^{-l}[\beta-\Delta_n,\beta)\subset W_n\vee x\in U_{n,k-j+1}.
	\]
	Thus in view of \eqref{qnlan}, \eqref{pdec} and the fact that $U_n$ and $W_n$ are disjoint we get
	\begin{equation}\label{jval}
		\begin{split}
	Leb\{x\in W_n\cup Z_n\cup U_n;&\, S_{iq_n}(f)(x)-iS_{q_n}(f)(0)=jD_\beta\}\\&=q_n\Delta_n+(n-1)s_1^n(\Delta_n-\de_n)\to2\g.
	\end{split}
	\end{equation}
	Moreover by \eqref{limitZn} and \eqref{qnlan} we get that 
	\begin{equation}\label{0val}
		\begin{split}
	Leb\{x\in W_n\cup Z_n\cup U_n;&\, S_{iq_n}(f)(x)-iS_{q_n}(f)(0)=0\}\\&\ge Leb(Z_n)=(q_n-s_1^n)\Delta_n\to\g.
	\end{split}
		\end{equation}
	On the other hand in view of \eqref{wn04} we have that
		\begin{equation}\label{ival}
			\begin{split}
				Leb\{x\in W_n\cup Z_n\cup U_n;&\, S_{iq_n}(f)(x)-iS_{q_n}(f)(0)=iD_{\beta}\}\\&\ge Leb(U_{n,K-i+1})=(n-1)s_1^n(\Delta_n-\de_n)\to\g.
			\end{split}
		\end{equation}	
		In view of \eqref{miaraXn} we also have
			\begin{equation}\label{wnunzn}
				\lim_{n\to\infty} Leb(W_n\cup Z_n\cup U_n)=\lim_{n\to\infty}Leb(I\setminus X_n)=1. 
			\end{equation}
		Thus combining \eqref{jval}, \eqref{0val}, \eqref{ival} and \eqref{wnunzn} we get that for every $i=1,\ldots,K$ there exists $\al,\beta\ge\g$ such that the following convergence holds
		
			\begin{equation}\label{maincocycle}\begin{split}
			\lim_{n\to\infty}&(S_{iq_n}(f)-iS_{q_n}(f)(0))_*Leb=P_i\quad \text{where}\\ &P_i:=\al\de_{0}+\beta\de_{iD_\beta}+
			2\g\sum_{0<j<i}\de_{jD_\beta}.
			\end{split}
			\end{equation}
	Since $L<K$, in view of Theorem \ref{main}, for $a_n=S_{q_n}(f)(0)$  we have the following convergences
	\[
	\lim_{n\to\infty}T^f_{Ka_n}\to\int_{\R}T_{-t}^f\,dP_K(t)
	\]
	and
	\[
	\lim_{n\to\infty}T^f_{La_n}\to\int_{\R}T_{-t}^f\,dP_L(t).
	\]
	Since $P_K$ and $P_L$ have different number of atoms and are supported on compact sets (hence with exponential decay) then so do rescalings of these measures. Thus Corollary \ref{cor:maintool} implies that the flows $\{S^{T,f}_{Kt}\}_{t\in\R}$ and $\{S^{T,f}_{Lt}\}_{t\in\R}$ are spectrally disjoint that is part \emph{(i)} of Theorem \ref{Main} is proved.

\begin{remark}
As a consequence of the first part of Theorem \ref{Main} together with the results included in \cite{DaRy} and Proposition 7.2 from \cite{BFr1}, one can show that in every connected component of the Moduli space of translation structures, the set of those translation structures for which the associated vertical flow is disjoint with all its natural rescalings, form a topologically generic set.
	\end{remark}
	
\begin{remark}
One can replace in Theorem \ref{Main} special flows with integral automorphisms and obtain analogous result for them. The proof of such result goes along the same lines as the proof of first part of Theorem \ref{Main}. As a consequence, one can show that if $\pi$ is a \emph{degenerate} permutation of alphabet $\mathcal A$ in the sense of Veech (see Section 5 in \cite{Veech}) then for almost every  $\la\in\Lambda^{\mathcal A}$ and for every distinct $K,L\in\N$, the automorphisms $T^K_{\pi,\la}$ and $T^L_{\pi,\la}$ are spectrally disjoint. This is a generalization of Theorem B.1 from \cite{CE}. The criterion on spectral disjointness that needs to be used in the case of automorphisms is similar but more elaborate, see Proposition 2 in \cite{ALR}.
\end{remark}

\subsection{Piecewise linear roof functions}
Now we prove the second part of Theorem \ref{Main} where instead of considering piecewise constant roof functions, we consider piecewise linear roof functions with constant non-zero slope. We do not assume however that there is an additional discontinuity point inside some of the exchanged intervals. As one can see this discontinuity was crucial in generating "asymmetry" in the limit measures required to obtain spectral disjointness. In this case this role is handed over to the non-zero slope. 

	Throughout this proof we rely on the notation and results of subsections \ref{sec:rec}, \ref{sec:const} and \ref{sec:rig}. Assume that $(\pi,\la)\in \Upsilon$ and $f$ is piecewise linear function with constant slope $\kappa\neq 0$, linear over intervals exchanged by $T:=T_{\pi,\la}$. In view of the result of the aforementioned subsections, we only need to evaluate of $S_{iq_n}(f)$ on $W_n\cup Z_n \cup U_n$, pick $a_n$ so that $S_{iq_n}(f)-ia_n$ is a bounded sequence on $W_n\cup Z_n\cup U_n$ and finally calculate the limit $\lim_{n\to\infty}(S_{iq_n}(f)-ia_n)_*Leb_{W_n\cup Z_n\cup U_n}$.

\subsubsection{\bf Values of the cocycle.}
Let $a_n=S_{q_n}(f)(0)$. First we evaluate the number $S_{iq_n}(f)(T^m(0))$ for $i=1,\ldots,K$ and $m=0,\ldots,q_n-1$. Recall that
\[
\Delta_n=\la_{\hat\pi_0^{-1}(d)}^{r_n}-\la_{\hat\pi_0^{-1}(d)}^{r_n}.
\]
Since $0\in W_n$, by \eqref{transWn} we have that for every $m=0,\ldots,q_n-1$ and $j=1,\ldots,K$ the following holds
\begin{equation}\label{jak0}
T^{jq_n}(T^m(0))=T^m(0)+j\Delta_n.
\end{equation}
Recall that there are no discontinuities of $T$ in the interior of $W_n$. Thus
\[
S_{q_n}(f)(T^{jq_n}(T^m0))=S_{q_n}(f)(T^m(0))+j\kappa q_n\Delta_n.
\]
On the other hand
\[
\begin{split}
S_{q_n}(f)(T^m(0))&=\sum_{k=m}^{q_n-1}f(T^i(0))+\sum_{k=0}^{m-1}f(T^{q_n+k}(0))\\
&=\sum_{k=m}^{q_n-1}f(T^i(0))+\sum_{k=0}^{m-1}f(T^{k}(0))+m\kappa\Delta_n\\
&=S_{q_n}(f)(0)+m\kappa\Delta_n.
\end{split}
\]
Thus we obtain that 
\begin{equation}\label{value0}
	\begin{split}
	S_{iq_n}(f)(T^m(0))&=\sum_{j=0}^{i-1}S_{q_n}(f)(T^{jq_n}(T^m(0)))\\&=\sum_{j=0}^{i-1}\big(S_{q_n}(f)(T^m(0))+j\kappa q_n\Delta_n \big)\\
	&=\sum_{j=0}^{i-1}\big(S_{q_n}(f)(0)+(jq_n+m)\kappa \Delta_n \big)\\&=iS_{q_n}(f)(0)+\Big(\frac{i(i-1)}{2}q_n+im\Big)\kappa\Delta_n
	\end{split}
	\end{equation}
for $i=1,\ldots,K$ and $m=0,\ldots,q_n-1$.

\textit{Points from $W_n$.}
Assume that $x\in W_n$ that is $x\in T^mJ^n$ for some $m=0,\ldots,q_n-1$. This implies that
\begin{equation}\label{xprzes}
x= T^{m}(0)+T^{-m}(x),
\end{equation}
where $T^{-m}(x)\in[0,|J^n|)$.
Since $x\in W_n$ then for $m=0,\ldots,q_n-1$ the points $T^{kq_n+m}x$ for every $k=0,\ldots,K-1$ belong to the same level of either tower $\{T^M I_{\hat\pi_0^{-1}(1)}^n;\ 0\le M<s_1^n\}$ or $\{T^M I_{\hat\pi_0^{-1}(d)}^n;\ 0\le M<s_d^n\}$. In particular, there are no discontinuities of $T$ between $T^{k+m}(0)$ and $T^k(x)$. Thus by \eqref{value0} and \eqref{xprzes} we get
\begin{equation}
	\begin{split}
	S_{iq_n}&(f)(x)=S_{iq_n}(f)(T^{m}(0))+iq_n\kappa\cdot T^{-m}(x)\\
	&=iS_{q_n}(f)(0)+\Big(\frac{i(i-1)}{2}q_n+im\Big)\kappa\Delta_n+iq_n\kappa\cdot T^{-m}(x)
	\end{split}
	\end{equation}
	for every $x\in W_n$.
	Thus we obtained that for every $m=0,\ldots,q_n-1$ we have
	\[
	\big(S_{iq_n}(f)-iS_{q_n}(f)(0)\big)_*Leb|_{T^{m}J^n}=g_{W_n,m}\,dx,
	\]
	where
	\[
	\displaystyle
	{g_{W_n,m}=\tfrac{1}{iq_n\kappa}\chi_{\left(\frac{i(i-1)}{2}q_n+im\right)\kappa\Delta_n+iq_n\kappa\left[0,|J^n|\right)}.}
	\]
	Recall also that in view of \eqref{1/10} we have that $K\Delta_n<|J^n|$.
	It follows that
		\begin{equation}\label{vNWn}
		\big(S_{iq_n}(f)-iS_{q_n}(f)(0)\big)_*Leb|_{W_n}=g_{W_n}\,dx,
		\end{equation}
		where
	\[
	\begin{split}
	g&_{W_n}=\sum_{m=0}^{q_n-2}\tfrac{m+1}{iq_n\kappa}\chi_{\left(\frac{i(i-1)}{2}q_n+im\right)\kappa\Delta_n+i\kappa[0,\Delta_n)}\\
	&+\tfrac{1}{i\kappa}\chi_{\frac{i(i-1)}{2}q_n\kappa\Delta_n+i\kappa [(q_n-1)\Delta_n,q_n|J^n|)}\\
	&+\sum_{m=1}^{q_n-1}\tfrac{q_n-m}{iq_n\kappa}\chi_{\left(\frac{i(i-1)}{2}q_n+i(m-1)\right)\kappa\Delta_n+iq_n\kappa|J^n|+i\kappa[0,\Delta_n)}.
	\end{split}
	\]

\textit{Points from $Z_n$.} Assume now that $x\in Z_n$, that is 
for some $m=0,\ldots,s_d^n+(n-1)s_1^n-1$ we have
\[
x\in T^m\big[\Sigma_n,\Sigma_n+\Delta_n\big)\subset T^mI^n_{\hat\pi_0^{-1}(d)},
\]
or in other words
\[
x=T^{m+s_1^n}(0)+\tilde x,\ \text{where}\ \tilde x=T^{q_n-s_1^n-m}x-\Delta_n.
\]
Then $\tilde x\in[-\Delta_n,0)$.
By \eqref{Znupper} and \eqref{Zntower}, for $M=0,\ldots,q_n-1$ we have that the points $T^{kq_n+M}x$ for every $k=0,\ldots,K-1$ belong to the same level of either tower $\{T^M I_{\hat\pi_0^{-1}(1)}^n;\ 0\le M<s_1^n\}$ or $\{T^M I_{\hat\pi_0^{-1}(d)}^n;\ 0\le M<s_d^n\}$. 
In particular, there are no discontinuity points of $T$ between $T^jx$ and $T^j(T^{m+s_1^n}(0))$ and by \eqref{transZn} we have $T^j(T^{m+s_1^n}(0))-T^jx=\tilde x$. Thus, by \eqref{value0}, we obtain
\begin{equation}\label{kocZn}
\begin{split}
S_{iq_n}(f)(x)&=S_{iq_n}(f)(T^{m+s_1^n}0)+iq_n\kappa\tilde x\\
&=iS_{q_n}(f)(0)+\Big(\frac{i(i-1)}{2}q_n+i(m+s_1^n)\Big)\kappa\Delta_n+iq_n\kappa\tilde x,
\end{split}
\end{equation}
for every $x\in Z_n$. 	It follows that
\begin{equation}\label{vNZn}
\big(S_{iq_n}(f)-iS_{q_n}(f)(0)\big)_*Leb|_{Z_n}=g_{Z_n}\,dx,
\end{equation}
where
\[
g_{Z_n}=\frac{1}{iq_n\kappa}\sum_{m=0}^{q_n-s_1^n-1}\chi_{\left(\frac{i(i-1)}{2}q_n+i(m+s_1^n)\right)\kappa\Delta_n+iq_n\kappa[-\Delta_n,0)}
\]

\textit{Points from $U_n$.} Suppose now that $x\in U_n$ that is $x\in U_{n,p}$ for some $1\le p\le K$. Then there exists $0\le m<(n-2)s_1^n$ such that 
\begin{equation}\label{vNdefm}
x\in T^m \big[|J^n|+(p-1)\Delta_n,|J^n|+p\Delta_n \big)\subset T^mI_{\hat\pi_0^{-1}(1)}^n,
\end{equation}
and \eqref{xprzes} holds.
In particular $x$ lies on the same level of tower $\{T^M I_{\hat\pi_0^{-1}(1)}^n;\ 0\le M<s_1^n\}$ or $\{T^M I_{\hat\pi_0^{-1}(d)}^n;\ 0\le M<s_d^n\}$ as $T^m(0)$. Moreover, if $k=0,\ldots,K-p-1$ and $M=0,\ldots,q_n-1$ then $T^{kq_n+M}(x)$ belongs to the same level of tower $\{T^M I_{\hat\pi_0^{-1}(1)}^n;\ 0\le M<s_1^n\}$ or $\{T^M I_{\hat\pi_0^{-1}(d)}^n;\ 0\le M<s_d^n\}$ as $T^{kq_n+M+m}(0)$. Thus if $i\le K-p$ then by \eqref{value0} and \eqref{xprzes} we have 
\begin{equation}\label{vNUn1}
	\begin{split}
		S_{iq_n}&(f)(x)=S_{iq_n}(f)(T^{m}(0))+iq_n\kappa\cdot T^{-m}(x)\\
		&=iS_{q_n}(f)(0)+\Big(\frac{i(i-1)}{2}q_n+im\Big)\kappa\Delta_n+iq_n\kappa\cdot T^{-m}(x).
	\end{split}
\end{equation}
As in the proof of \eqref{wn02}, by using \eqref{glue0} and \eqref{glue}, we get that 
\begin{equation}\label{vNUn5}
	\begin{split}
		S_{(K-p+1)q_n}(f)(x)&=S_{(K-p+1)q_n}(f)(T^{m}(0))+(K-p+1)q_n\kappa\cdot T^{-m}(x)\\
		&=(K-p+1)S_{q_n}(f)(0)\\
		&\ +\Big(\frac{(K-p+1)(K-p)}{2}q_n+(K-p+1)m\Big)\kappa\Delta_n\\
		&\ +(K-p+1)q_n\kappa\cdot T^{-m}(x),
	\end{split}
\end{equation}
that is \eqref{vNUn1} is satisfied also for $i=K-p+1$. 
Since by \eqref{pdec} we have $T^{-m}x\in[|J^n|+(p-1)\Delta_n+\de_n,|J^n|+p\Delta_n)$, we in particular obtained that for $p\le K-i+1$ the following holds
		\begin{equation}\label{vNUn0}
			\big(S_{iq_n}(f)-iS_{q_n}(f)(0)\big)_*Leb|_{U_{n,p}}=g_{U_{n,p}}\,dx,
		\end{equation}
		where in view of \eqref{vNdefm} and \eqref{pdec}, $g_{U_{n,p}}$ is defined as follows
\[
g_{U_{n,p}}=\frac{1}{iq_n\kappa}\sum_{m=0}^{q_n-s_1^n-1}\chi_{\left(\frac{i(i-1)}{2}q_n+im\right)\kappa\Delta_n+iq_n\kappa(|J^n|+(p-1)\Delta_n)+iq_n\kappa[\delta_n,\Delta_n)}
\]

Recall that $T^{(K-p+1)q_n}(x)\in Z_n$ (see e.g. \eqref{przeskok2}). More precisely $T^{(K-p+1)q_n}(x)\in T^m[\Sigma_n,\Sigma_n+\Delta_n)$. Moreover in view of \eqref{kocZn} we have
\begin{equation}\label{vNUn2}
		\begin{split}
			&S_{(i-(K-p+1))q_n}(f)(T^{(K-p+1)q_n}(x))-(i-(K-p+1))S_{q_n}(f)(0)\\& =\!\Big(\!\frac{(i-(K-p+1))(i-(K-p+2))}{2}q_n\!+\!(i-(K-p+1))(m+s_1^n)\!\Big)\kappa\Delta_n\\ &\quad +(i-(K-p+1))q_n\kappa\tilde x,
		\end{split}
	\end{equation}
where $\tilde x\in[-\Delta_n,0)$ is such that 
\begin{equation}\label{vNUn3}
T^{(K-p+1)q_n}(x)=T^{m+s_1^n}(0)+\tilde x.
\end{equation}
Note that in view of \eqref{transUn1} and \eqref{przeskok} we have
\[
T^{(K-p+1)q_n}(T^{-m}x)=T^{-m}x+(K-p+1)\Delta_n.
\]
Thus by \eqref{vNUn3} we have
\[
\tilde x=T^{(K-p+1)q_n}(T^{-m}x)-T^{s_1^n}(0)=T^{-m}x+(K-p+1)\Delta_n-\Sigma_n\Delta_n=T^{-m}x+(K-p)\Delta_n-\Sigma_n.
\] 
Hence by \eqref{vNUn2} we get
\begin{equation}\label{vNUn4}
	\begin{split}
		&S_{(i-(K-p+1))q_n}(f)(T^{(K-p+1)q_n}(x))-(i-(K-p+1))S_{q_n}(f)(0)\\& =\!\Big(\!\frac{(i-(K-p+1))(i-(K-p+2))}{2}q_n\!+\!(i-(K-p+1))(m+s_1^n)\!\Big)\kappa\Delta_n\\ &\quad +(i-(K-p+1))q_n\kappa\big(T^{-m}x+(K-p)\Delta_n-\Sigma_n\big).
	\end{split}
\end{equation}
By combining \eqref{vNUn5} with \eqref{vNUn4} we get that if $p>K-i+1$ and $x\in U_{n,p}$ then
\[
\begin{split}
S_{iq_n}(f)(x)&-iS_{q_n}(f)(0)=(S_{(K-p+1)q_n}(f)(x)-(K-p+1)S_{q_n}(f)(0))\\
&\quad +(S_{(i-(K-p+1))q_n}(f)(T^{(K-p+1)q_n}(x))-(i-(K-p+1))S_{q_n}(f)(0))\\
&=\big(\frac{1}{2}i(i-1)q_n-(i-(K-p+1))q_n+im)\big)\kappa\Delta_n\\
&\quad +iq_n\kappa T^{-m}x-
(i-(K-p+1))q_n\Sigma_n+(i-(K-p+1))s_1^n\kappa\Delta_n.
\end{split}
\]
By \eqref{pdec} we have that 
\[
T^{-m}(x)\in\big[|J^n|+(p-1)\Delta_n+\delta_n,|J^n|+p\Delta_n \big).
\]
Hence we have that for $p>K-i+1$ the following holds
\begin{equation}\label{vNUn}
	\big(S_{iq_n}(f)-iS_{q_n}(f)(0)\big)_*Leb|_{U_{n,p}}=g_{U_{n,p}}\,dx,
\end{equation}
where
\[
g_{U_{n,p}}=\frac{1}{iq_n\kappa}\sum_{m=0}^{q_n-1}\chi_{C_{n,p,m}},
\]
and 
\[
\begin{split}
C_{n,p,m}&=\big(\frac{1}{2}i(i-1)q_n-(i-(K-p+1))q_n+im)\big)\kappa\Delta_n\\
&\quad +iq_n\kappa(|J^n|+(p-1)\Delta_n)-
(i-(K-p+1))q_n\Sigma_n\\
&\quad+(i-(K-p+1))s_1^n\kappa\Delta_n+iq_n\kappa[\delta_n,\Delta_n)
%\Big[\Big(\frac{i(i-1)}{2}q_n+im\Big)\kappa\Delta_n+iq_n\kappa(|J^n|+(p-1)\Delta_n)\\
%&\ -s_1^n\kappa\Delta_n\ -(i-(K-p+1))q_n\kappa\Delta_n\\
%&\ -(i-(K-p+1))q_n\kappa\sum_{b\neq\hat\pi_0^{-1}(d)}\la^{r_n}_{b}+iq_n\kappa[,\\
%&\ \Big(\frac{i(i-1)}{2}q_n+im\Big)\kappa\Delta_n+iq_n\kappa(|J^n|+p\Delta_n)\\
%&\ -s_1^n\kappa\Delta_n\ -(i-(K-p+1))q_n\kappa\Delta_n\\
%&\ -(i-(K-p+1))q_n\kappa\sum_{b\neq\hat\pi_0^{-1}(d)}\la^{r_n}_{b}
% \Big),]
\end{split}
\]
for $m=0,\ldots,q_n-1$.

\subsubsection{\bf The density of limit measure.} We now use the results from the previous subsection to evaluate the limit $P_i=\lim_{n\to\infty}(S_{iq_n}(f)-iS_{q_n}(f)(0))_*Leb_{W_n\cup U_n\cup Z_n}$. First note that in view of \eqref{wnunzn} we have
\[
\begin{split}
\lim_{n\to\infty}&(S_{iq_n}(f)-iS_{q_n}(f)(0))_*Leb_{W_n\cup U_n\cup Z_n}\\
&=\lim_{n\to\infty}\big((S_{iq_n}(f)-iS_{q_n}(f)(0))_*Leb|_{W_n}+(S_{iq_n}(f)-iS_{q_n}(f)(0))_*Leb|_{Z_n}\\
&\ +\sum_{p=1}^K(S_{iq_n}(f)-iS_{q_n}(f)(0))_*Leb|_{U_{n,p}}\big).
\end{split}
\]
Recall that $q_n=s_d^n+ns_1^n$.  Thus by \eqref{qnlan} and \eqref{pdec} we have that
\[
\lim_{n\to\infty}Leb(U_{n,p})=\lim_{n\to\infty}q_n\Delta_n=\lim_{n\to\infty}(q_n-s_1^n)\Delta_n=\lim_{n\to\infty}Leb(Z_n)=\g.
\]
By the above convergence, \eqref{qnlan} and by \eqref{wnunzn} we also obtain
\[
\lim_{n\to\infty}q_n|J^n|=\lim_{n\to\infty}Leb(W_n)=1-\lim_{n\to\infty}Leb(\bigcup_{p=1}^KU_{n,p}\cup Z_n)=1-(K+1)\g.
\]
Finally in view of \eqref{miaraXn} we also get 
\[
\lim_{n\to\infty}q_n\Sigma_n=\lim_{n\to\infty}(W_n\cup\bigcup_{p=1}^KU_{n,p})=1-\g
\]
and by \eqref{pdec} and \eqref{qnlan} we have
\[
\lim_{n\to\infty}q_n\delta_n=0.
\]
We use the above convergences to compute the density of $P_i$.

In view of \eqref{vNWn} we have 
\begin{equation}\label{summ1}
\lim_{n\to \infty}\big((S_{iq_n}(f)-iS_{q_n}(f)(0))\big)_*Leb|_{W_n}=g_W\,dx,
\end{equation}
where
\[
\begin{split}
g&_W(x)=\frac{1}{i\kappa}\Big(\frac{1}{i\kappa\g}x-\frac{i-1}{2}\Big)\chi_{\frac{i(i-1)}{2}\kappa\g+i\kappa[0,\g)}(x)\\
&\ + \frac{1}{i\kappa}\chi_{\frac{i(i-1)}{2}\kappa\g+i\kappa[\g,1-(K+1)\g)}(x)\\
&\ + \frac{1}{i\kappa}\Big(-\frac{1}{i\kappa\g}x+\frac{i-1}{2}+\frac{1}{\g} -K\Big)\chi_{\frac{i(i-1)}{2}\kappa\g+i\kappa(1-(K+1)\g)+i\kappa[0,\g)}(x).
\end{split}
\]
Secondly, in view of \eqref{vNZn}, we obtain
\begin{equation}\label{summ2}
\lim_{n\to \infty}\big((S_{iq_n}(f)-iS_{q_n}(f)(0))\big)_*Leb|_{Z_n}=g_Z\,dx,
\end{equation}
where
\[
\begin{split}
g_Z&=\frac{1}{i\kappa}\Big(\frac{1}{i\kappa\g}x-\frac{i-1}{2}+1\Big)\chi_{\frac{i(i-1)}{2}\kappa\g+i\kappa[-\g,0)}(x)\\
&\ +\frac{1}{i\kappa}\Big(-\frac{1}{i\kappa\g}x+\frac{i-1}{2}+1\Big)\chi_{\frac{i(i-1)}{2}\kappa\g+i\kappa[0,\g)}(x).
\end{split}
\]
Finally, we obtain that for $p=1,\ldots,K$ we have that
\begin{equation}\label{summ3}
\lim_{n\to \infty}\big((S_{iq_n}(f)-iS_{q_n}(f)(0))\big)_*Leb|_{U_{n,p}}=g_{U,p}\,dx,
\end{equation}
where the formula for $g_{U,p}$ depends on the relation of $p$ and $i$. More precisely for $p\le K-i+1$ in view of \eqref{vNUn0} we have
\[
\begin{split}
g_{U,p}&=\frac{1}{i\kappa}\Big(\frac{1}{i\kappa\g}x-\frac{i-1}{2}-\frac{1}{\g} +K-p+2\Big) \cdot\chi_{\frac{i(i-1)}{2}\kappa\g+i\kappa(1-(K+1)\g)+ip\kappa\g+i\kappa[-\g,0)}(x)\\
&\ +\frac{1}{i\kappa}\Big(-\frac{1}{i\kappa\g}x+\frac{i-1}{2}+\frac{1}{\g} -K+p\Big)\cdot\chi_{\frac{i(i-1)}{2}\kappa\g+i\kappa(1-(K+1)\g)+ip\kappa\g+i\kappa[0,\g)}(x).
\end{split}
\]
On the other hand, if $i> 1$ then for $p> K-i+1$ in view of \eqref{vNUn} we have
\[
\begin{split}
g&_{U,p}=\frac{1}{i\kappa}\Big(\frac{1}{i\kappa\g}x-\frac{i-1}{2}-\frac{1}{\g} +K-p+2+\frac{p-(K-i+1)}{i\g}\Big)\\
&\cdot\chi_{\frac{i(i-1)}{2}\kappa\g+i\kappa(1-(K+1)\g)+ip\kappa\g-(p-(K-i+1))\kappa+i\kappa[-\g,0)}(x)\\
&+\frac{1}{i\kappa}\Big(-\frac{1}{i\kappa\g}x+\frac{i-1}{2}+\frac{1}{\g} -K+p-\frac{p-(K-i+1)}{i\g}\Big)\\
&\cdot\chi_{\frac{i(i-1)}{2}\kappa\g+i\kappa(1-(K+1)\g)+ip\kappa\g-(p-(K-i+1))\kappa+i\kappa[0,\g)}(x).
\end{split}
\]
	\begin{figure}
		\includegraphics[scale=1.5]{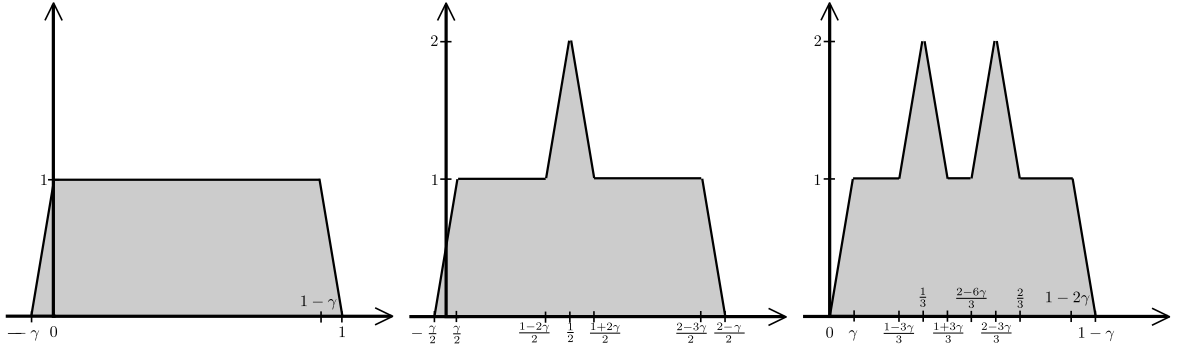}
		$\quad\quad\quad\quad$
		\caption{Exemplary densities of measures $P_1$, $Res_2(P_2)$ and $Res_3(P_3)$, from left to right, for $\kappa=1$. }
	\end{figure}
As a consequence, by \eqref{summ1}, \eqref{summ2} and \eqref{summ3} we obtain that for $i=1,\ldots,K$ we have
\[
\lim_{n\to\infty}(S_{iq_n}(f)-iS_{q_n}(f)(0))_*Leb_{W_n\cup U_n\cup Z_n}=P_i=G_i\,dx,
\]
where
\begin{equation}\label{vN2}
\begin{split}
G_1(x)&=\frac{1}{\kappa}\Big(\frac{1}{\kappa\g}x+1\Big)\chi_{[-\kappa\g,0)}(x)+\frac{1}{\kappa}\chi_{[0,\kappa(1-\g))}(x)\\&\quad+\frac{1}{\kappa}\Big(-\frac{1}{\kappa\g}x+\frac{1}{\g}\Big)
\chi_{[\kappa(1-\g),\kappa)}(x),
\end{split}
\end{equation}
whereas for $i>1$ we have
\begin{equation*}
\begin{split}
G_i(x)&=\frac{1}{i\kappa}\Big(\frac{1}{i\kappa\g}x-\frac{i-1}{2}+1\Big)\chi_{\frac{i(i-1)}{2}\kappa\g+i\kappa[-\g,0)}(x)\\
&\ +\frac{1}{i\kappa}\chi_{\frac{i(i-1)}{2}\kappa\g+i\kappa[0,1-i\g)}\\
&\ +\frac{1}{i\kappa}\Big(-\frac{1}{i\kappa\g}x-\frac{i+1}{2}+\frac{1}{\g} +1\Big)
\chi_{\frac{i(i-1)}{2}\kappa\g+i\kappa[1-i\g,1-(i-1)\g)}(x)\\
&\ +\Big(\sum_{p=1}^{i-1} \frac{1}{i\kappa}\Big(\frac{1}{i\kappa\g}x-\frac{i-1}{2} +p+1-\frac{p}{i\g}\Big)\\
&\ \cdot\chi_{\frac{i(i-1)}{2}\kappa\g+p\kappa-ip\kappa\g+i\kappa[-\g,0)}(x)\\
&\ +\frac{1}{i\kappa}\Big(-\frac{1}{i\kappa\g}x+\frac{i-1}{2} -(p-1)+\frac{p}{i\g}\Big)\\
&\ \cdot\chi_{\frac{i(i-1)}{2}\kappa\g+p\kappa-ip\kappa\g+i\kappa[0,\g)}(x)\Big).
\end{split}
\end{equation*}
Then $Res_i(P_i)$ is an absolutely continuous measure with a density
\begin{equation}\label{vN1}
\begin{split}
	\tilde G_i(x)&=\frac{1}{\kappa}\Big(\frac{1}{\kappa\g}x-\frac{i-1}{2}+1\Big)\chi_{\frac{i-1}{2}\kappa\g+\kappa[-\g,0)}(x)\\
	&\ +\frac{1}{\kappa}\chi_{\frac{i-1}{2}\kappa\g+\kappa[0,1-i\g)}\\
	&\ +\frac{1}{\kappa}\Big(-\frac{1}{\kappa\g}x-\frac{i+1}{2}+\frac{1}{\g} +1\Big)
	\chi_{\frac{i-1}{2}\kappa\g+\kappa[1-i\g,1-(i-1)\g)}(x)\\
	&\ +\Big(\sum_{p=1}^{i-1} \frac{1}{\kappa}\Big(\frac{1}{\kappa\g}x-\frac{i-1}{2} +p+1-\frac{p}{i\g}\Big)\\
	&\ \cdot\chi_{\frac{i-1}{2}\kappa\g+\frac{p}{i}\kappa-p\kappa\g+\kappa[-\g,0)}(x)\\
	&\ +\frac{1}{\kappa}\Big(-\frac{1}{\kappa\g}x+\frac{i-1}{2} -(p-1)+\frac{p}{i\g}\Big)\\
	&\ \cdot\chi_{\frac{i-1}{2}\kappa\g+\frac{p}{i}\kappa-p\kappa\g+\kappa[0,\g)}(x)\Big).
\end{split}
\end{equation}
It is worth to mention that the shape of the graph of the above density which comes from the points belonging to $W_n$, $Z_n$ and $U_{n,p}$ for $p\le K-i+1$ is a trapezoid of height $\frac{1}{\kappa}$ while the impact of the points from $\bigcup_{p=K-i+2}^{K}U_{n,p}$ comes in the form of triangles also of height $\frac{1}{\kappa}$. The position of those triangles in the graph of $\tilde G_i$ is determined by the term $\frac{p}{i}\kappa$ in the sum which appears in the formula for the aforementioned density. Moreover the shape of the graph of $\tilde G_1:=G_1$ is just a trapezoid of height $\frac{1}{\kappa}$. In particular, the number of points for which the value $\frac{2}{\kappa}$ is admitted by $\tilde G_i$ (which are the maximums of this function for $i>1$) is equal $(i-1)$. Since $K<L$, then it follows that $\tilde G_K\neq\tilde G_L$ and thus $Res_K(P_K)\neq Res_L(P_L)$. Hence Corollary \eqref{cor:maintool} implies that flows $\{T_{Kt}^{f}\}_{t\in\R}$ and $\{T_{Lt}^{f}\}_{t\in\R}$ are spectrally disjoint.

\end{proof}


\begin{thebibliography}{9}

	\bibitem{ALR}E. H.\ El Abdalaoui, M.\ Lemańczyk, T.\ de la Rue, \emph{On spectral disjointness of powers for rank-one transformations and Möbius orthogonality}. J. Funct. Anal. 266 (2014), no. 1, 284–317.
	
	\bibitem{AF}A.\ Avila, G.\  Forni,
	\emph{Weak  mixing  for  interval  exchange  transformations  and  translation  flows}, Ann. of Math. 165 (2007), 637-664.
	

	\bibitem{BFr}P.\ Berk, K.\ Fr\k{a}czek, \emph{On special flows that are not isomorphic to their inverses},  Discrete\ Contin.\ Dyn.\ Syst. 35 (2015), 829–-855.
	
		
		\bibitem{BFr1}P.\ Berk, K.\ Frączek, T.\ de la Rue, \emph{On typicality of translation flows which are disjoint with their inverse}, published online in Journal of the Institute of Mathematics of Jussieu, \url{arXiv:1703.09111}.
	
	\bibitem{BSZ}J.\ Bourgain, P.\ Sarnak, T.\ Ziegler,{\em Disjointness of M{\"o}bius from horocycle flows}, from Fourier
analysis and number theory to Radon transforms and geometry, Dev. Math., vol. 28, Springer, New
York, 2013, 67-83.

	
	\bibitem{CE}J.\ Chaika, A.\ Eskin, \emph{Mobius disjointness for interval exchange transformations of three intervals}, \url{arXiv:1606.02357}.
	
	\bibitem{DaRy} A.I.\ Danilenko, V.V.\ Ryzhikov,
	\emph{On self-similarities of ergodic flows},
	Proc.\ Lond.\ Math.\ Soc.\ 104 (2012), 431-454.
	
	\bibitem{dJun} A.\ del Junco,
	\emph{Dsjointness of measure-preserving transformations, minimal self-joinings and category.}
	Ergodic theory and dynamical systems, I (College Park, Md., 1979-80), pp.\ 81-89,
		Progr.\ Math.\ 10, Birkhäuser, Boston, Mass. (1981).
	
	
	\bibitem{FKL} S.\ Ferenczi, J.\ Kulaga-Przymus, M.\ Lema{\'n}czyk, {\em Sarnak's conjecture: what's new}, arXiv: 1710.04039.
	
	
	\bibitem{FrKLem3}  K.\ Frączek, J.\ Kułaga-Przymus; M.\ Lemańczyk, \emph{On the self-similarity problem for Gaussian-Kronecker flows}. Proc. Amer. Math. Soc. 141 (2013), no. 12, 4275–4291. 
	\bibitem{FrKuLem} K.\ Fr\k{a}czek,\ J. Ku\l aga-Przymus,\ M. Lema\'nczyk,
	\emph{Non-reversibility
	and self-joinings of higher orders for ergodic flows},
	J.\ Anal.\ Math.\ 122 (2014), 163-227.
	\bibitem{FrLem5} K.\ Frączek, M.\ Lemańczyk, \emph{On disjointness properties of some smooth flows}, Fund. Math. 185 (2005), No.2, 117-142 
	 \bibitem{FrLem3} K.\ Frączek, M.\ Lemańczyk, \emph{On the self-similarity problem for ergodic flows}. Proc. Lond. Math. Soc. (3) 99 (2009), no. 3, 658–696.
	\bibitem{FrLem2} K.\ Fr\k{a}czek, M.\ Lema\'nczyk, \emph{On symmetric logarithm and some old examples
	in smooth ergodic theory}, Fund. Math., 180(3),241-255, 2003.
	\bibitem{FrLem4} K.\ Frączek, M.\ Lemańczyk, \emph{Smooth singular flows in dimension 2 with the minimal self-joining property}, Monatsh. Math. 156 (2009), 11-45
	\bibitem{KLU} A.\ Kanigowski, M.\ Lemańczyk, C.\ Ulcigrai, \emph{On disjointess properties of some parabolic flows}, \url{arXiv:1810.11576}.
	\bibitem{Kul} J.\ Kułaga-Przymus, \emph{On the self-similarity problem for smooth flows on orientable surfaces}. Ergodic Theory Dynam. Systems 32 (2012), no. 5, 1615–1660.
	\bibitem{LW} M.\ Lemańczyk, M.\ Wysokińska, \emph{On analytic flows on the torus which are disjoint from systems of probability origin}, Fundamenta Math. 195 (2007), 97-124.
	\bibitem{Mas}H.\ Masur, \emph{Interval exchange transformations and measured foliations.} Ann.\ of Math.\ (2) 115 (1982), 169–200.
	\bibitem{Ra} G.\ Rauzy, \emph{\'Echanges d'intervalles et transformations induites},
	Acta Arith.\ 34 (1979), 315-328.
	\bibitem{UlSin}  Y.\ Sinai, C.\ Ulcigrai, \emph{Renewal-type Limit Theorem for the Gauss Map and Continued Fractions},  Ergodic Theory Dynam. Systems 28 (2008), no. 2, 643–655. 
	\bibitem{Veech} W.A.\ Veech, \emph{Gauss measures for transformations on the space of interval exchange maps.} Ann.\ Math.\ 115 (1982), 201–242.
	\bibitem{Veech0} W.A.\ Veech, \emph{Interval exchange transformations}, J. Anal.\ Math.\ 33 (1978), 222-272.
	\bibitem{Viana} M.\ Viana, \emph{Ergodic theory of interval exchange maps}, Rev.\ Mat.\ Complut.\ 19 (2006), 7-100.
	\bibitem{CU} C.\ Ulcigrai, {\em Absence of mixing in area-preserving flows on surfaces},  Ann. of Math. (2), 173 (2011), 1743-1778.
	
	

\end{thebibliography}
\end{document}